\documentclass{article}
\usepackage{graphicx} % Required for inserting images
\usepackage{amsfonts}
\usepackage[utf8]{inputenc}
\usepackage{xcolor}
\usepackage{etoolbox}
\usepackage{amssymb}
\usepackage{tikz-cd}
\usepackage{amscd,amsmath}
\usepackage{amssymb}
\usepackage{amsthm}
\usepackage{enumerate}
\usepackage{mathptmx} % font Times New Roman (simile)
\usepackage{hyperref}

\newcounter{claim}[section]
\AtEndEnvironment{proof}{\setcounter{claim}{0}} 
\newenvironment{claim}[1][]{\refstepcounter{claim}\par\medskip\noindent\textit{Claim~\theclaim:}\space#1}

\def\Ind#1#2{#1\setbox0=\hbox{$#1x$}\kern\wd0\hbox to 0pt{\hss$#1\mid$\hss}
\lower.9\ht0\hbox to 0pt{\hss$#1\smile$\hss}\kern\wd0}

\def\notind#1#2{#1\setbox0=\hbox{$#1x$}\kern\wd0
\hbox to 0pt{\mathchardef\nn=12854\hss$#1\nn$\kern1.4\wd0\hss}
\hbox to 0pt{\hss$#1\mid$\hss}\lower.9\ht0 \hbox to 0pt{\hss$#1\smile$\hss}\kern\wd0}

\newtheorem{theorem}{Theorem}[section]
\newtheorem{corollary}[theorem]{Corollary}

\theoremstyle{definition}
\newtheorem{definition}[theorem]{Definition}
\newtheorem*{notation}{Notation}
\newtheorem*{CZC}{Cherlin--Zilber Algebraicity Conjecture}
\newtheorem{lemma}[theorem]{Lemma}

\newtheorem{proposition}[theorem]{Proposition}
\newtheorem{remark}[theorem]{Remark}
\newtheorem{question}{Question}[section]
\newenvironment{claimproof}[1]{\par\noindent\textit{Proof:}\space#1}{\hfill $\blacksquare$\medskip\par}

\def\aut{\operatorname{Aut}}
\def\ker{\operatorname{ker}}
\def\RM{\operatorname{RM}}
\def\tp{\operatorname{tp}}
\def\isog{\approxeq}
\def\B{\mathfrak B}
\def\A{\mathfrak A}
\def\ann{\operatorname{Ann}}
\def\im{\operatorname{im}}

\title{Bisoluble and Binilpotent Skew Braces of Finite Morley Rank}
\author{M. Ferrara -- M. Invitti -- M. Trombetti -- F.O. Wagner}

\begin{document}
\newenvironment{claimm}{\par\medskip\noindent\textit{Claim:}\space}{\par}

\maketitle

\begin{abstract}
\noindent A connected skew brace of finite Morley rank is left nilpotent (resp. weakly soluble) whenever the additive and multiplicative groups are nilpotent (resp. are soluble and the left chief length is at most $3$).
\end{abstract}

\section{Introduction}
In 2007, Rump \cite{rump2007braces} defined the notion of {\em brace} (a generalization of radical rings) to construct certain set-theoretic solutions of the Yang--Baxter equation, a very important functional equation in mathematical physics and other mathematical fields. Guarnieri and Vendramin introduced skew braces in \cite{guarnieri2017skew}, extending the brace framework to the case of non-abelian additive groups.

\begin{definition}
A \emph{skew brace} is a set $B$ endowed with two group laws $+$ and $\circ$ such that the \emph{skew left distributivity law} holds: $$a\circ (b+c)=a\circ b-a+a\circ c$$ for all $a,b,c\in B$. Moreover, if $\mathfrak X$ is any class of groups and the additive group $B_+\in\mathfrak X$, then $B$ is said to be {\em of type $\mathfrak X$}. In this terminology, Rump's braces are just skew braces of abelian type. If both the additive and the multiplicative group are $\mathfrak X$, we call a skew brace {\em bi-$\mathfrak X$}.\end{definition}
 
Since then, the interest in skew braces has increased, not only in relation to the Yang--Baxter equation, but also as algebraic structures in their own right. However, most of the research has concentrated on finite skew braces.
 
The model-theoretic study of infinite skew braces is a very recent research direction, reflecting the relative novelty of skew braces themselves. The foundations of a model-theoretic approach were laid in \cite{braces}. As in the case of groups, skew braces are studied under various additional model-theoretic tameness conditions which ensure that the class under consideration is reasonably well-behaved. Of course, since skew braces carry two group structures, a good model-theoretic understanding of the relevant groups is required. 
 
In this paper we shall study solubility and nilpotency of skew braces under the assumption of finite Morley rank, a model-theoretic tameness condition inspired by algebraic geometry and implying a notion of dimension on definable sets, namely Morley rank (see Definition \ref{d:fMR}).
 
Not much is known about the mutual influence of the additive and multiplicative
groups of a skew brace, or about how the two group structures constrain the full skew brace structure. For example, it is apparently unknown if there exists a skew brace whose multiplicative group is nilpotent but whose additive group is not soluble (see~\cite{adv}). 

Our results contribute to this question by linking additive and multiplicative nilpotency to brace solubility and nilpotency (see Definition \ref{d:s,n}). Recall that a group is {\em connected} if it has no definable subgroup of finite index; by \cite[Theorem 3.7]{braces}, in a skew brace of finite Morley rank (or more generally in a {\em stable} skew brace) additive and multiplicative connectivity coincide.  We shall show:\begin{enumerate}
\item A connected binilpotent skew brace of finite Morley rank is left nilpotent.
\item A connected bisoluble skew brace of finite Morley rank is weakly
 soluble, provided its left chief length (Definition \ref{d:lcl}) is at most~$3$.
\end{enumerate}
In particular, connected skew braces of Morley rank $2$ are soluble, and connected bisoluble skew braces of Morley rank $3$ are weakly soluble. 

It should be noted that left nilpotency of a binilpotent skew brace has also been shown under the assumptions of finiteness \cite[Theorem 4.8]{cedo2019skew} or $\mathfrak{M}_c$ and $\omega$-categoricity \cite{braces}, suggesting a possible general behaviour for tame skew braces (see Section \ref{sec:6}).

\subsection{Layout of the paper}
In Section \ref{sec:2}, we recall the preliminary concepts and results necessary for the analysis of skew braces of finite Morley rank. The only new result is Lemma \ref{Kequi}, which is fundamental in the analysis of bisoluble and binilpotent skew braces.

In Section \ref{sec:3}, we introduce and study the concept of left chief series. This is the equivalent of the chief series for groups in the context of skew braces. We prove a version of the Jordan--H\"older theorem for skew braces. This allows us to define the \emph{left chief length} as the length of a left chief series, a concept that is extensively used in our study of bisoluble and binilpotent skew braces.

Section \ref{sec:4} is devoted to the analysis of bisoluble connected skew braces of finite Morley rank and left chief length at most $3$, and Section \ref{sec:5} to the study of binilpotent skew braces of finite Morley rank.

Finally, Section \ref{sec:6} contains some concluding remarks and open questions.

\section{Preliminaries}\label{sec:2}
This section aims to provide the reader with the necessary definitions and fundamental properties of skew braces and groups of finite Morley rank needed for the rest of the paper.

\subsection{Skew braces}
In this section, we recall the basic properties of skew braces. Moreover, we prove some easy facts that will later be useful. Recall from the introduction that a ({\em left}) {\em skew brace} $(B,+,\circ)$ is a set $B$ with two group operations $+,\circ$ such that, for all $a,b,c\in B$, the {\em skew left distributivity} law holds:
$$a\circ(b+c)= a\circ b-a+a\circ c.$$
Note that this implies that the additive and multiplicative unit coincide; we denote it by $0$.
We usually denote a skew brace simply by $B$, and its additive and multiplicative
groups by $B_+$ and $B_\circ$, respectively.
For $b\in B$ we define the function $\lambda_b:B\to B$ as
$$\lambda_b:a\mapsto \lambda_b(a):=-b+b\circ a.$$ 
It follows immediately from skew left distributivity that $\lambda_b$ is an additive automorphism of $B$. Moreover, we define the map $\lambda:B\to\aut(B_+)$ as follows:
$$\lambda:b\mapsto \lambda_b\in \aut(B_+).$$
It is immediate to verify that $\lambda$ is a homomorphism from $B_\circ$ to $\aut(B_+)$. If $a,b\in B$, then we also define the $\ast$-operation as follows: 
$$a\ast b=-a+a\circ b-b=\lambda_a(b)-b.$$
This operation measures the discrepancy between $+$ and $\circ$, since $a\ast b=0$ if and only if $\lambda_a(b)=b$ if and only if $a+b=a\circ b$. We recall some useful properties of $\ast$:
\begin{enumerate}[(i)]
	\item $a\ast (b+c)=a\ast b+b+a\ast c-b$;
	\item $(a\circ b)\ast c=a\ast (b\ast c)+b\ast c+a\ast c$.
\end{enumerate}

A subset $X$ of $B$ is {\em $\lambda$-invariant} if $\lambda_b(X)\subseteq X$ for all $b\in B$. Note that this also holds for $b^{-1}$, so $X=\lambda_b\lambda_{b^{-1}}(X)\subseteq\lambda_b(X)$ and we must have equality. Thus $X$ is $\lambda$-invariant if and only if $b\circ X=b+X$ for all $b\in B$. In particular, a $\lambda$-invariant set $X$ is an additive subgroup if and only if it is a multiplicative subgroup.
 
A {\em skew sub-brace} of $B$ is a subset $A$ which is both an additive and a multiplicative subgroup. A skew sub-brace $A$ of $B$ is a {\em left ideal } if it is $\lambda$-invariant. A left ideal $A$ of $B$ is a {\em strong left ideal} if it is additively normal in $B$; a strong left ideal is an {\em ideal} if it is also multiplicatively normal. In this case the induced structure on the quotient $B/A$ is again a skew brace (recall that by $\lambda$-invariance, additive and multiplicative cosets coincide); conversely the kernel of any skew brace homomorphism is an ideal and the isomorphism theorems hold. Note that if $A$ is a (strong) left ideal in $B$ and $L$ is an additively characteristic subgroup of $A$, then $L$ is again a (strong) left ideal.

Given two strong left ideals $L$ and $M$ of a skew brace $B$, we can consider the additive quotient $L/M$. We generally omit the subscript $+$ since the additive quotient does not carry a multiplicative structure.

\begin{notation} If $B$ is a brace, we denote the semi-direct product naturally associated with it by the same letter but in gothic typeface: $\B:=B_+\rtimes_{\lambda} B_\circ$. If $A\le B$ is a sub-brace, $\mathfrak A=A_+\rtimes_\lambda A_\circ$ denotes the corresponding subgroup, and $\mathfrak A_0:=A_+\rtimes_\lambda\{0\}$. The same holds for sub-braces $L,M,\ldots$ of $B$: We have $\mathfrak L$, $\mathfrak L_0$, $\mathfrak M$ and $\mathfrak M_0$, etc., all subgroups of~$\B$.\end{notation} 

The previous definitions may now be reformulated in the semidirect product.

\begin{lemma}\label{NormalG}
Let $B$ be a skew brace, and $A\subseteq B$ a subset.
\begin{itemize}
\item The action of $\B$ on $B_+$ by conjugation is given by 
$(b,b_1)\cdot a=b+\lambda_{b_1}(a)-b$.
\item $A$ is a skew sub-brace if and only if $\mathfrak{A}$ is a subgroup of $\B$.
\item $A$ is a strong left ideal if and only if $\mathfrak{A_0}:=A_+\times \{0\}$ is a normal subgroup of $\B$, if and only if $A$ is an additive subgroup invariant under the action of $\B$ on $B_+$.
\item $A$ is an ideal if and only if $\mathfrak{A}$ is a normal subgroup in $\B$.
\end{itemize}
\end{lemma}
\begin{proof}
Obvious.
\end{proof}

\begin{lemma}
Let $B$ be a skew brace. Then $K=\ker(\lambda)$ is a skew sub-brace of $B$.
\end{lemma}
\begin{proof}
Of course, $K$ is a (normal) subgroup of $B_\circ$. If $x,y\in K$, then $x+y=x\circ y\in K$. Moreover, if $a\in K$ and $b\in B$, then
$$0=(a-a)\ast b=(a\circ(-a))\ast b=a\ast ((-a)\ast b)+(-a)\ast b+a\ast b=(-a)\ast b.$$ Thus, $-a\in K$ and hence $K$ is a skew sub-brace.
\end{proof}

When working with group-theoretic concepts (such as centraliser, normaliser etc.), we shall use super- and subscripts to indicate whether they are intended additively or multiplicatively.

\begin{lemma}
Let $B$ be a skew brace and $H\subseteq L$ strong left ideals of $B$. Then $C_B^+(L/H)$ is a strong left ideal of $B$.
\end{lemma}
\begin{proof} Clearly $C_B^+(L/H)$ is a normal additive subgroup; it is $\lambda$-invariant since $\lambda$ is an additive automorphism and both $L$ and $H$ are $\lambda$-invariant.
\end{proof}

The following result gives some equivalent conditions for a subset to be an ideal. 
\begin{lemma}\label{lemuseful}
Let $B$ be a skew brace and $X$ a subset of~$B$. If any three of the following four properties hold, then $X$ is an ideal:
\begin{enumerate}[(i)]
\item $X$ is a normal additive subgroup of $B$.
\item $X$ is a normal multiplicative subgroup of $B$.
\item $X$ is $\lambda$-invariant.
\item $X\ast B\subseteq X$.
\end{enumerate} 
\end{lemma}
\begin{proof}
If $X$ satisfies (i)--(iii), then $X$ is an ideal by definition. If $X$ satisfies (i), (ii) and~(iv), then for $x\in X$ and $b\in B$ there is $x'\in X$ such that $b\circ x=x'\circ b$, so
$$\lambda_b(x)=-b+b\circ x=-b+x'\circ b=-b+x'-x'+x'\circ b-b+b\in -b+x'+X+b=X.$$
Hence (iii) holds.
	
Suppose $X$ satisfies (i), (iii) and (iv). Then $X$ is a strong left ideal, and for $x\in X$ nd $b\in B$ we have $x\circ b\in x+X+b=X+b=b+X$, so
$$b^{-1}\circ x\circ b\in b^{-1}\circ(b+X)=b^{-1}\circ b-b^{-1}+b^{-1}\circ X=\lambda_{b^{-1}}(X)=X,$$
so $X$ is multiplicatively normal.
	
Finally, if $X$ satisfies (ii)--(iv), recall that (ii) and (iii) imply that $X$ is an additive subgroup. Consider $b\in B$ and $x\in X$. By (iii), there is $x'\in X$ such that $\lambda_b(x')=x$, while by (ii) there is $x''\in X$ such that $b^{-1}\circ x''\circ b=x'$. Thus, 
$$-x''+b+x-b=-x''+b+\lambda_b(x')-b=-x''+x''\circ b-b=x''\ast b\in X$$ by (iv), and hence $b+x-b\in X$, which means that $X$ is additively normal.
\end{proof}

The following lemma will be useful in Sections \ref{sec:4} and \ref{sec:5}.
\begin{lemma}\label{WelDefstar}
Let $B$ be a skew brace, $b\in B$, and $M_1\unlhd L_1$, $M_2\unlhd L_2$ be additive subgroups. Suppose $b\ast L_1\subseteq L_2$, $b\ast M_1\subseteq M_2$, $L_1\le N^+_B(M_2)$ and $L_1\le C_B^+(L_2/M_2)$. Then the map $\ast_b:x+M_1\mapsto b\ast x+M_2$ is a well-defined homomorphism from $L_1/M_1$ to $L_2/M_2$.
\end{lemma}
\begin{proof}For $\ell,\ell'\in L_1$ and $m\in M_1$ we have:
$$\begin{aligned}b\ast (\ell+\ell'+m)&=b\ast\ell+\ell+b\ast(\ell'+m)-\ell=b\ast\ell+\ell +b\ast\ell'+\ell'+b\ast m-\ell'-\ell\\
&\in b\ast\ell+\ell +b\ast\ell'+\ell'+M_2-\ell'-\ell=b\ast\ell+\ell +b\ast\ell'-\ell+M_2\\
&=b\ast\ell+b\ast\ell'+M_2,\end{aligned}$$
so the map is well-defined and additive.	
\end{proof}

Now, let $B$ be a skew brace. If $+_{{op}}$ denotes the opposite operation of $+$, then $B^{op}=(B,+_{op},\circ)$ (or $B^{opp}$ as denoted in \cite{koch2023skew}) is easily seen to be a skew brace too. 
The $\lambda$-function of~$B^{op}$ is usually denoted by $\mu$, so 
$\mu:B_\circ\to \aut(B_+)$ with $\mu_b(a)=b\circ a-b$. Clearly, any (normal) subgroup of $B_+$ is again a (normal) subgroup in $B_{+_{op}}$, and a normal additive subgroup is $\lambda$-invariant if and only if it is $\mu$-invariant. It follows easily that skew sub-braces and (strong left) ideals in $B$ and $B^{op}$ coincide, and a skew brace homomorphism $B\to A$ is also a skew brace homomorphism $B^{op}\to A^{op}$. However, left ideals of $B$ and $B^{op}$ may differ.

Let $A=A_+$ be an additive quotient of two strong left ideals of a skew brace $B$. Then $\lambda_b$ induces an automorphism $\lambda_b^A$ of $A_+$ for every $b\in B$, and $\lambda$ induces a homomorphism $\lambda^A:B_\circ\to\aut(A_+)$. We even have a natural action of $\B$ on $A_+$ (which combines the $\lambda$-action with additive conjugation).

\begin{lemma}
Let $B$ be a skew brace, and $A$ a finite strong left ideal of $B$. If $B$ is additively connected, then $A\subseteq Z^+(B)$ and $\ker(\lambda)=\ker\big(\lambda^{B/A}\big)$; if $B$ is multiplicatively connected, then $B\ast A=\{0\}$.
\end{lemma}
\begin{proof} Suppose first that $B$ is additively connected. By additive normality, the additive conjugacy class of every $a\in A$ is contained in $A$, whence finite. It follows that $C^+_B(a)$ is a definable additive subgroup of finite index in $B$, and must equal $B$ by additive connectivity. Thus $A\le Z^+(B)$.

Clearly $\ker(\lambda)\le\ker\big(\lambda^{B/A}\big)$. For $b\in\ker\big(\lambda^{B/A}\big)$ consider the map $\ast_b:x\mapsto b\ast x$. Since $b\ast B\subseteq A$ and $A$ is additively central, $\ast_b$ is a definable additive endomorphism of $B$ with finite image. So $\ker(\ast_b)$ is a definable additive subgroup of finite index and must equal $B$ by additive connectivity. It follows that $b\in\ker(\lambda)$.

Suppose now that $B$ is multiplicatively connected. Since $A$ is finite, $\aut(A_+)$ is finite, so $\ker\lambda^A$ has finite index in $B_\circ$. As $\ker(\lambda^A)$ is definable, it must equal $B$ by multiplicative connectivity. Thus $B\ast A=\{0\}$.
\end{proof}

Because skew braces have two group structures (and, as we have seen, there are actually three operations involved), there is more than one concept of solubility and nilpotency for skew braces. We refer the interested reader to \cite{tutti23,ballester,braces} for an account of these, while we restrict ourselves to defining the ones we need in the context of braces of soluble/nilpotent type. For two skew sub-braces $A$ and $A'$ we define the {\em brace commutator group} of $A$ and $A'$ as 
$$[A,A']:=\langle [A,A']_+,A\ast A'\rangle,$$
where $\langle X\rangle$ is the skew sub-brace generated by $X$.
\begin{remark}\label{r:commutator}Let $L$ be a strong left ideal in $B$, and recall that $\B=B_+\rtimes_\lambda B_\circ$, and $\mathfrak L_0=(L,+)\times\{0\}$. Then
$$([B,L],+)\times\{0\}=[\mathfrak L_0,\B]$$
(where the second commutator group is the group commutator in $\B$).\end{remark}

\begin{definition}\label{d:s,n} Let $B$ be a skew brace. We call $B$ {\em weakly soluble} (resp., ({\em strongly left}) {\em soluble} {\em of soluble type}) if it has an {\em abelian} series (resp. {\em abelian} ({\em strong left}) {\em ideal} series), that is a series of skew sub-braces (resp., (strong left) ideals) 
$$B=L_0\ge L_1\ge\cdots\ge L_n=\{0\}$$
such that $[L_i,L_i]\subseteq L_{i+1}$ for all $i<n$. The minimal length of such a series is the {\em (weak/strong left) derived length} of $B$.

We call $B$ {\em left nilpotent of nilpotent type} if there exists a {\em strong left central series}, that is, a series of strong left ideals 
$$B=L_0\ge L_1\ge\cdots\ge L_n=\{0\}$$
such that $[B,L_i]\subseteq L_{i+1}$ for all $i<n$. The minimal length of such a series is the {\em strong left nilpotency class} of $B$.\end{definition}

\begin{theorem}[{\cite[Theorem 5.13]{braces}}]\label{t:lnp-Bnp}
A skew brace $B$ is left nilpotent of nilpotent type if and only if $\B$ is a nilpotent group.
\end{theorem}

As in the case of groups, there is a minimal left central series, the {\em lower left central series} defined inductively by $\Gamma_1(B)=B$ and $\Gamma_{n+1}(B)=[B,\Gamma_n(B)]$. The following lemma shows that every term $\Gamma_i(B)$ is in fact a strong left ideal of $B$. 

\begin{lemma}
Let $B$ be a skew brace. If $L$ is a strong left ideal of $B$, then $[B,L]$ is a strong left ideal of~$B$.  Moreover, if $I$ is an ideal of $B$ containing $L$, then $[I,L]\supseteq [L,L]_\circ$. In particular, $[B,B]$ is an ideal of $B$.
\end{lemma}
\begin{proof} By Remark \ref{r:commutator} we have
$$[B,L]_+\times\{0\}=[\mathfrak L_0,\B],$$
which is normal in $\B$. Hence $[B,L]$ is a strong left ideal by Lemma \ref{NormalG}.
	
For the moreover part, consider $l,l'\in L$. Then $l\ast l'\in [I,L]$ and so $l\circ l'=l+l'+[I,L]$. Since $[I,L]$ contains also $[I,L]_+\supseteq [L,L]_+$, it contains $[L,L]_\circ$.
\end{proof}

\begin{remark}\label{r:ws-fe}If $B$ is a weakly soluble skew brace, then the additive exponent is finite  (resp.\ a power of a prime $p$) if and only if the multiplicative exponent is finite  (resp.\ a power of $p$): In an abelian series, all quotients are trivial braces, so have the same additive and multiplicative exponent.\end{remark}

\subsection{Groups of finite Morley rank}
Morley rank was introduced in \cite{morley1965categoricity} to study totally transcendental theories. 
\begin{definition}\label{d:fMR}
Let $T$ be a complete theory. Define {\em Morley rank} to be the least function $\RM$ from the collection of non-empty definable sets (of finite arity) in models of $T$ to the ordinals together with $\infty$ (where $\infty=\infty+1$ is greater than any ordinal) satisfying for any definable set $X$ and ordinal $\alpha$:
\begin{quote}$\RM(X)\geq \alpha+1$ if and only if in some elementary extension there exists a family $(X_i:i<\omega)$ of definable pairwise disjoint subsets of $X$ such that $\RM(X_i)\geq \alpha$ for every $i\in\omega$.\end{quote}
A theory $T$ has \emph{finite Morley rank} if $\RM(x=x)$ is finite.\end{definition}
If $X$ is defined by a formula $\varphi(\bar x,\bar a)$, then $\RM(X)=\RM(\varphi(\bar x,\bar a))$ only depends on $\varphi$ and the type $\tp(\bar a)$. In particular, Morley rank is preserved under automorphism. If $T$ has finite Morley rank, then {\em every} formula (in arbitrarily many variables) has finite Morley rank.

One of the main examples of a theory of finite Morley rank is the theory of algebraically closed fields of given characteristic. As a consequence, any algebraic group over an algebraically closed field is a group of finite Morley rank. The converse, for simple groups, is the main open question about groups of finite Morley rank.

\begin{CZC}
Let $G$ be a simple definable group of finite Morley rank. Then $G$ is definably isomorphic to an algebraic group over an algebraically closed field.  
\end{CZC}

If $H\le G$ are groups of finite Morley rank, then $\RM(G/H)=\RM(G)-\RM(H)$. In particular $\RM(G)=\RM(H)$ if and only if the index $|G:H|$ is finite. It follows that a group of finite Morley rank satisfies the descending chain condition on definable subgroups. In particular the intersection of all definable subgroups of finite index of $G$ equals a finite subintersection, again of finite index, the {\em connected component} $G^0$, and $G$ is {\em connected} if $G=G^0$. Note that the image of a connected group under a definable homomorphism is again connected, and if $G/H$ and $H$ are connected, so is $G$.

By \cite[Theorem 3.7]{braces}, if $B$ is a skew brace of finite Morley rank, then the additive and the multiplicative connected component coincide and form a definable ideal $B^0$, the {\em connected component} of $B$. As for groups we say that a skew brace $B$ is {\em connected} if $B^0=B$.
\smallskip

\noindent{\bf Convention.} In the context of finite Morley rank, all groups, braces, ideals, morphisms, etc.\ under consideration will be supposed to be definable. Nevertheless, we may still add the adjective {\em definable} for emphasis.\smallskip

We now recall the main results about groups of finite Morley rank. An infinite group is {\em minimal} if it has no definable infinite proper subgroup. In particular, a minimal group is always infinite and connected; our set-up almost exclusively deals with connected groups and largely ignores finite sections. A connected group of Morley rank $1$ is minimal.
\begin{lemma}[\cite{Reineke}]
A minimal group of finite Morley rank is abelian.
\end{lemma}
Next, we shall consider group actions. The action of a group $G$ on an abelian group $A$ is {\em irreducible} if $A$ has no definable infinite $G$-invariant proper subgroup; it is {\em faithful} if only the identity of $G$ fixes all of $A$. If $G$ acts irreducibly on $A$, we also say that $A$ is {\em $G$-minimal}.

\begin{theorem}[{Zilber Linearisation Theorem \cite[Theorem 3.7]{poizat}}]\label{ZilberLin}
Let $G$ be an abelian group acting faithfully and irreducibly on a connected abelian group $A$, all definable in a theory of finite Morley rank. Then there exists a definable field $K$ of finite Morley rank such that $A$ is isomorphic to the additive group $K_+$ and $G$ embeds into the multiplicative group $K^{\times}$. 
\end{theorem}

Theorem \ref{ZilberLin} highlights the importance of definable fields in the analysis of groups of finite Morley rank. A well-known theorem of Macintyre states that a field of finite Morley rank is algebraically closed (see \cite[Theorem 3.1]{poizat}). Thus, all fields in this paper will be algebraically closed. 

In the soluble case, we have the following result.
\begin{theorem}\label{TriActG'}
In a structure of finite Morley rank, let $G$ be a connected soluble group acting irreducibly on a connected abelian group $A$. Then $G'$ acts trivially on $A$.
\end{theorem}
Thus we may apply Theorem \ref{ZilberLin} to the action of $G/C_G(A)$ on $A$.

A key corollary of Theorem \ref{TriActG'} is the Lie--Kolchin--Malcev Theorem for groups of finite Morley rank.
\begin{corollary}[{\cite[Corollary 3.19]{poizat}}]\label{NilpDerG}
Let $G$ be a connected soluble group of finite Morley rank. Then $G'$ is nilpotent.
\end{corollary}

Linearisation also yields the existence of torsion elements (the roots of unity of the field), often in some section (as in Theorem \ref{TriActG'}). We shall need the following lifting theorem for Sylow subgroups.

\begin{lemma}[\cite{pw00}]\label{SylowLifting}
Let $G$ be a soluble group of finite Morley rank and $H$ a definable normal subgroup of $G$. Then the Sylow $p$-subgroups of $G/H$ are the images of the Sylow $p$-subgroups of $G$.
\end{lemma}

The next theorem follows from \cite[Theorem 2]{cherlin}.
% and \cite[Fact 2.4 and Corollary 2.7]{wiscons}.

\begin{theorem}\label{structuremorleyrank2}
A connected group of Morley rank $2$ is soluble.
\end{theorem}

Another important theorem is the Zilber Indecomposability Theorem. We shall only need the following consequence:
\begin{theorem}(See \cite[Theorem 3.13]{poizat}) Let $G$ be a group of finite Morley rank, $H$ a definable connected subgroup, and $X$ any set. Then the group $[H,X]$ is definable and connected.\end{theorem}
In particular, for a connected group of finite Morley rank, the groups in the derived series and in the lower central series are all definable. Moreover, if $B$ is a connected brace and $L$ a strong left ideal, then $[B,L]$ is connected and definable by Remark \ref{r:commutator}.
\smallskip

We now turn to nilpotent groups of finite Morley rank. We shall call a definable normal subgroup $H$ of $G$ \emph{minimal normal} if it has no definable infinite proper subgroup normal in $G$.

Recall that a group is {\em radicable} if the map $x\mapsto x^n$ is surjective for all $n$. (For abelian groups, this is divisibility.)
\begin{theorem}[{see \cite[Theorem 6.8]{borovik}}]\label{t:nesin}
Let $G$ be a connected nilpotent group of finite Morley rank. Then $G$ is a central product $G=D\star F$, where  $F$ has finite exponent, $D$ is radicable, and both are definable, connected and characteristic. In particular $D\cap F$ is finite.
\end{theorem}
\begin{theorem}\label{sylow}
In a nilpotent radicable group of finite Morley rank, the Sylow $p$-subgroup is a finite extension of a direct product of finitely many Pr\"ufer $p$-groups. In particular it is abelian-by-finite.
\end{theorem}

\begin{corollary}
Let $G$ be a connected nilpotent group of finite Morley rank. The following are equivalent:
\begin{enumerate}[(i)]
\item $G$ is radicable.
\item $G$ has no infinite subgroup of finite exponent.
\item Every connected definable section of $G$ is radicable.
\item No connected definable section of $G$ has an infinite subgroup of finite exponent.
\end{enumerate}\end{corollary}
\begin{proof}
Suppose that $G$ is radicable, and let $H$ be an infinite subgroup of finite exponent. As $H$ is a direct product of its Sylow subgroups, we may assume that $H$ is a $p$-group for some prime $p$. By Theorem \ref{sylow} it is contained in a finite extension of a direct product of finitely many Pr\"ufer $p$-groups. But such a product does not contain an infinite subgroup of finite exponent. This shows $(i)\Rightarrow(ii)$.

Conversely, suppose that $G$ does not have an infinite subgroup of finite exponent. Then $F=\{1\}$ in the decomposition of Theorem \ref{t:nesin}, so $G=D$ is radicable. This shows $(ii)\Rightarrow(i)$.

Clearly $(iii)\Rightarrow(i)$ and $(iv)\Rightarrow(ii)$. Moreover $(iii)\Leftrightarrow(iv)$ follows from $(i)\Leftrightarrow(ii)$.

Now suppose $G$ is radicable. Then $G$ does not have an infinite subgroup of finite exponent, and neither does any definable subgroup. So a connected definable subgroup $H$ is radicable. If $N$ is a definable normal subgroup of $H$, the quotient $H/N$ remains radicable. This shows $(i)\Rightarrow(iii)$.	
\end{proof}
	
\begin{lemma}\label{center}
Let $G$ be a connected nilpotent group and $H\unlhd G$ an infinite normal subgroup. Then $H\cap Z(G)$ is infinite. In particular, every minimal normal subgroup $H$ of $G$ is central. 
\end{lemma}
\begin{proof}
Let $m$ be minimal such that $H\cap Z_m(G)$ is infinite. If $h\in H\cap Z_m(G)$, then $[G,h]\subseteq Z_{m-1}(G)\cap H$ which is finite. Hence $C_G(h)$ has finite index in $G$, and must be all of $G$ by connectivity. It follows that $H\cap Z_m(G)= Z(G)\cap H$ is infinite.
\end{proof}

\begin{lemma}[Normaliser Condition]\label{l:NC} Let $G$ be a nilpotent group of finite Morley rank, and $H\le G$ a subgroup of infinite index. Then the index $|N_G(H):H|$ is infinite.\end{lemma}
\begin{proof} Let $m$ be maximal such that $Z_m(G)^0\le H$. Now for every $g\in G$ the map $x\mapsto [g,x]+Z_m(G)^0$ is a definable homomorphism from $Z_{m+1}(G)^0$ to $Z_m(G)/Z_m(G)^0$ with finite image, so its kernel is a definable subgroup of finite index and must equal $Z_{m+1}(G)^0$. It follows that $[Z_{m+1}(G)^0,G]\le H$. Hence $Z_{m+1}(G)^0\le N_G(H)$, and $$|N_G(H):H|\ge |Z_{m+1}(G)^0H:H|$$ is infinite.\end{proof} 

\begin{lemma}\label{lemnotbothfe}
Let $G$ be a group and $H$ a normal subgroup such that both $G/H$ and $H$ are connected abelian, and $H$ has no definable proper $G$-normal subgroup. Suppose that $H$ does not have finite exponent, and either\begin{enumerate}[(i)]
\item $H\le Z(G)$ and $G/H$ has no proper infinite normal subgroup, or
\item $G/H$ has finite exponent.
\end{enumerate}
Then $G$ is abelian.
\end{lemma}
\begin{proof}Note that connectedness of $G/H$ and $H$ implies connectedness of $G$.
\begin{enumerate} [(i)]
\item Assume $G$ not abelian, and take $g\in G\setminus Z(G)$. Since $G/H$ is abelian and $H\le Z(G)$, the function $\phi:x+H\mapsto [g,x]$ is a definable homomorphism from $G/H$ to $H$ with connected image. As $g$ is not central, $\ker(\phi)$ is a proper normal subgroup of $G/H$ and must be finite, so there is $n<\omega$ such that $g^n\in H$. But $\im(\phi)\le Z(G)$ is normal in $G$ and infinite, so $\im(\phi)=H$. Then $0=[g^n,x]=[g,x]^n$ for all $x\in G$, so $H$ has exponent $n$, a contradiction. 
\item The action of $G/H$ on $H$ by conjugation is irreducible, whence trivial by Lemma \ref{Kequi}, so $H\leq Z(G)$. As $G/H$ has finite exponent, there is $n$ such that $g^n\in H$ for all $g\in G$. We now finish as in part~(i).\qedhere
\end{enumerate}  
\end{proof}

The following definition allows us to avoid talking about finite index subgroups and quotients by a finite normal subgroup (see also Lemma~\ref{Kequi}).

\begin{definition} Two connected groups $G$ and $H$ are {\em isogenous} if there exist finite normal subgroups $G_0\unlhd G$ and $H_0\unlhd H$ such that $G/G_0\cong H/H_0$. This is clearly an equivalence relation, which we denote by $G\isog H$. Moreover, we say that $G$ is {\em isogenically embeddable} into $H$ if $G$ is isogenous to a subgroup of $H$.

We call $G$ and $H$ {\em definably isogenous} if there is a definable isogeny between $G$ and $H$. Again this is an equivalence relation. 
\end{definition}
In the finite Morley rank context, isogenies and isogenic embeddings will usually be definable.

\begin{remark}\label{isogenies}
Let $H$ be isogenically embeddable into $G$.
\begin{enumerate}[(i)]
\item If $G$ is of finite exponent, then $H$ is of finite exponent.
\item If $G$ has finite torsion, then $H$ has finite torsion.
\end{enumerate}
\end{remark}

\begin{remark}
Let $G$ and $H$ be definably isogenous. Then $\RM(G)=\RM(H)$; moreover $G$ is minimal if and only if $H$ is minimal.
\end{remark}

The following result partially prohibits the existence of an infinite group of finite Morley rank isogenically embeddable in both the additive and multiplicative group of definable fields of finite Morley rank. Even if it does not cover all possible cases, it is enough for our purposes.

\begin{lemma}\label{Kequi}
Let $K$ and $F$ be fields of finite Morley rank. Then:
\begin{enumerate}[(i)]
\item $K^{\times}$ cannot be isogenically embedded into $F_+$.
\item No infinite group of finite exponent is isogenous to a definable section of $K^{\times}$.
\item No infinite definable subgroups of $K^{\times}$ and of $F_+$ can be isogenous, provided that either $K$ or $F$ has positive characteristic.
\item No infinite definable subgroups of $K_+$ and of $K^{\times}$ can be definably isogenous.
\end{enumerate}
Note that the isogenies in (i)--(iii) need not be definable.
\end{lemma}
\begin{proof} As both $K$ and $F$ are algebraically closed, $K^\times$ has infinite exponent and infinite torsion, whereas $F_+$ has either finite exponent or no torsion.\begin{enumerate}[(i)]
\item This follows from Remark \ref{isogenies}.
\item Let $G$ be an infinite group of finite exponent $n$ isogenous to a definable section $D/L$ of $K^\times$. Then $D/L$ has finite exponent, say $\ell$, and the homomorphism $x\mapsto x^\ell$ from $D$ to $L$ has infinite kernel since $\RM(D)>\RM(L)$. But $x^\ell=1$ has at most $\ell$ solutions in $K$, a contradiction.
\item Suppose otherwise, and let $G$ be an infinite definable subgroup of $K^{\times}$ isogenous to a definable subgroup of $F_+$. If $F$ has positive characteristic, $G$ has finite exponent and we are done by (ii). If $K$ has positive characteristic, then \cite[Fact~2.30]{ALTINEL2004} implies that $G$ has infinite torsion and infinite exponent, and so does $F_+$ by Remark \ref{isogenies}, a contradiction.
\item Suppose otherwise. Then there is an infinite definable subgroup $G$ of $K_+$ definably isogenous to a definable subgroup $H$ of $K^\times$. If $\RM(G)=\RM(H)=\RM(K)$, then $K_+\isog K^{\times}$, contradicting (i). Hence $\RM(G)<\RM(K)$, so $K$ has a proper infinite definable additive subgroup. By \cite[Corollary 3.3]{poizat} the characteristic of $K$ is positive, so $G$ has finite exponent, contradicting (ii).
\qedhere\end{enumerate}\end{proof}

\section{Left chief series}\label{sec:3}
A fundamental tool to study the structure of binilpotent and bisoluble skew braces $B$ is the {\em left chief series}. We start by introducing the concept of $\Sigma$-composition series for a connected $\Sigma$-group $G$. 

\begin{definition} Let $G$ be a connected group, and $\Sigma$ a group of automorphisms of $G$. Then $G$ is said to be a {\em $\Sigma$-group}. A definable subgroup $H$ of $G$ such that $\sigma(H)=H$ for all $\sigma\in\Sigma$ is a {\em $\Sigma$-subgroup} of $G$. We shall call $G$ {\em $\Sigma$-simple} if it is infinite and has no infinite proper definable normal $\Sigma$-subgroup.
	
A {\em $\Sigma$-composition series} of $G$ is a finite series of connected $\Sigma$-subgroups
$$G=G_0\ge G_1\ge\cdots\ge G_n=\{1\}$$
such that $G_{i+1}$ is normal in $G_i$ and $G_i/G_{i+1}$ is $\Sigma$-simple for all $i<n$.

As usual, the quotients $G_i/G_{i+1}$ are the {\em factors} of the series, while the subgroups $G_i$ are the {\em terms} of the series. Two $\Sigma$-composition series $(G_i:i\leq n)$ and $(H_i:i\leq m)$ are {\em equivalent} if $n=m$ and there exists a permutation $s\in\operatorname{Sym}(n)$ such that $G_i/G_{i+1}$ is definably isogenous to $H_{s(i)}/H_{s(i)+1}$ for all $i<n$.
\end{definition}
\begin{remark}If $\Sigma$ is trivial, a $\Sigma$-composition series is just a composition series. If $\Sigma$ is the group of inner automorphisms of $G$, a $\Sigma$-composition series is called a {\em chief series}.
	
Note that in the finite Morley rank context, if all automorphisms in $\Sigma$ are definable, then the connected component of a normal $\Sigma$-subgroup is again a normal $\Sigma$-subgroup. Hence the connectivity requirement for a $\Sigma$-composition series follows from $\Sigma$-simplicity of the factors.\end{remark}

\begin{lemma}[Jordan--H\"older Theorem]\label{Jordan-Holder}
Let $G$ be a connected $\Sigma$-group of finite Morley rank. Then every subnormal series of connected $\Sigma$-subgroups of $G$ can be refined to a $\Sigma$-composition series, and any two $\Sigma$-composition series of $G$ are equivalent.
\end{lemma}
\begin{proof} Any subnormal series of connected $\Sigma$-subgroups of $G$ has its length bounded by $\RM(G)$, so a maximal refinement of the given series exists. By maximality any quotient is $\Sigma$-simple, so it is a $\Sigma$-composition series.
	
We show equivalence by induction on $\RM(G)$ and assume the theorem holds for groups of Morley rank strictly less than $\RM(G)$. Suppose two $\Sigma$-composition series of $G$ have a common term $H$ other than $G$ or $\{1\}$ such that $H$ is normal in $G$. Since both $\RM(G/H)$ and $\RM(H)$ are strictly less than $\RM(G)$, the induced $\Sigma$-composition series of $G/H$ must be equivalent, as must those of $H$. But then the original series are equivalent.

Now consider two $\Sigma$-composition series $\mathcal G$ and $\mathcal H$ of $G$ with first terms $G_1$ and $H_1$, respectively, and note that both $G_1$ and $H_1$ are normal in $G$. If $G_1=\{1\}$ then $G$ is $\Sigma$-simple, so $H_1=\{1\}$ and the series are equal. 

So suppose $G_1$ is infinite, as is $H_1$ since $G$ is not $\Sigma$-simple. Then $(G_1\cap H_1)^0$ is a connected $\Sigma$-subgroup normal in $G$ and contained in $G_1$ and in $H_1$. If $(G_1\cap H_1)^0=\{1\}$, then $G/G_1\isog H_1$ and $G/H_1\isog G_1$ definably, so we are done. Otherwise consider two quasi Jordan-H\"older series $\mathcal G'$ passing through $G_1$ and $(G_1\cap H_1)^0$, and $\mathcal H'$ passing through $H_1$ and $(G_1\cap H_1)^0$. By the second paragraph, $\mathcal G$ is equivalent to $\mathcal G'$, which is equivalent to $\mathcal H'$, which is equivalent to $\mathcal H$. It follows that $\mathcal G$ and $\mathcal H$ are equivalent.
\end{proof}

Now, we go back to the skew brace context.

\begin{definition}\label{d:lcl}
Let $B$ be a skew brace of finite Morley rank and $L$ a connected strong left ideal of $B$. Then $L_+$ can be seen as a $\B$-group, and a {\em $B$-left chief series} of $L$ is just a $\B$-composition series $(L_i:i\leq n)$ of the $\B$-group $L$. In other words, $(L_i:i\leq n)$ is a chain of connected strong left ideals of $B$ such that $L_0=L$, $L_n=\{0\}$, and
$L_i/L_{i+1}$ is a $\B$-simple additive subgroup of $B/L_{i+1}$ for $i<n$.

The (additive) factors $L_i/L_{i+1}$ are the {\em $B$-left chief factors} of $L$. The {\em $B$-left chief length} of $L$ is the length of a $B$-left chief series of $L$, which is well defined by Lemma \ref{Jordan-Holder}. A $B$-left chief series is {\em subideal} if $L_{i+1}$ is an ideal of $L_i$ for every $i<n$.
\end{definition} 
When $B=L$, we omit the $B$ and just write {\em left chief series} and {\em left chief factors}. Note that {\em a priori} left chief factors admit a $\lambda$-action, but have no multiplicative structure. In particular we shall consider them as additive $\B$-groups (unless indicated otherwise).

Note that in the brace context, as $\B$ contains additive conjugation, $\B$-simplicity is the same as $\B$-minimality (i.e.\ no infinite proper $\Sigma$-subgroup).

Having left chief length $1$ is equivalent to saying that $B$ has no definable infinite strong left ideal of infinite index.

We now determine the properties of $B$-left factors in the bisoluble, binilpotent and left nilpotent context. We start with the bisoluble case.

\begin{lemma}\label{Minimality}
Let $B$ be a connected bisoluble skew brace of finite Morley rank, and $L/M$ a left chief factor of $B$. Then $L/M$ is abelian, and $\B'$ acts trivially on it. Moreover, $L/M$ is either minimal or of prime exponent; if $A<B$ is an infinite connected skew sub-brace and $M<I<L$ is an $\mathfrak A$-invariant additive connected subgroup, then $\mathfrak A$ acts trivially on $L/M$.
\end{lemma}
\begin{proof}
$L/M$ is soluble, so $(L/M)'=(L_+'+M)/M<L/M$. Moreover $L'_+$ is a definable connected strong left ideal of $B$. By $\B$-minimality of the left chief factor $L_+'\le M$, so $L/M$ is abelian.

Note that $\B$ is soluble by bisolubility of $B$. The action of $\B$ on $L/M$ is irreducible, so $\B'$ acts trivially by Theorem \ref{TriActG'}. Moreover, if the action of $\B$ is non-trivial, then $L/M$ is isomorphic to the additive group of an algebraically closed field $K$ by Theorem \ref{ZilberLin}, so either $L/M$ is of prime exponent, or $\operatorname{char}(K)=0$ and $L/M$ is minimal. If the $\B$-action is trivial, any subgroup of $L/M$ is $\B$-invariant, whence either finite or equal to $L/M$ by $\B$-minimality. Thus $L/M$ is minimal.

For the moreover statement, note that since $I/M$ is a proper infinite subgroup of $L/M$, we cannot be in the minimal case, so the action of $\B$ on $L/M$ induces a field structure $K$ of positive characteristic.

Suppose now that the action of $\A$ on $L/M$ is not trivial, so $C_\B(L/M)\not\ge\A$. Then $I/M< L/M$ corresponds to an infinite subgroup $H< K_+$ invariant under the subgroup of $K^\times$ corresponding to $\A C_\B(L/M)/C_\B(L/M)\le \B/C_\B(L/M)$; note that this quotient is infinite by connectedness of $A$, whence of $\A$. It follows that $\{x\in K:xH\le H\}$ is an infinite proper definable subring (and in fact subfield by stability) of $K$, contradicting finiteness of Morley rank, as a proper pair of algebraically closed fields has Morley rank at least $\omega$.
Therefore the action of $\mathfrak A$ on $L/M$ must be trivial.
\end{proof}

\begin{lemma}\label{l:finmin=minfin}
Let $B$ be a connected bisoluble skew brace of finite Morley rank, and $L_1/L_2$ and $L_2/L_3$ two consecutive left chief factors. If $L_1/L_2$ has finite exponent and $L_2/L_3$ is divisible, then $L_1/L_3$ is abelian and there is an infinite connected strong $B$-left ideal $L_1>F>L_3$ such that $F/L_3$ has finite exponent and $L_1/F$ is divisible.\end{lemma}
\begin{proof}
For ease of notation we assume $L_3=\{0\}$. Now
$L_1$ is abelian by Lemma \ref{lemnotbothfe}, so $L_1=D+F$ by Theorem \ref{t:nesin}, where both $D$ and $F$ are connected strong $B$-left ideals, $D$ is divisible and $F$ has finite exponent. In fact, if $n$ is the exponent of $L_1/L_2$, we consider the endomorphism $\phi_n:x\mapsto n\cdot x$ of $L_1$. Then $D=\im(\phi_n)=L_2$ by divisibility and our choice of $n$, and $F=\ker(\phi_n)^0$ is a connected strong left ideal. Moreover $D\cap F$ is finite, and $$\RM(D+F)=\RM(D)+\RM(F)=\RM(\im(\phi_n))+\RM(\ker(\phi_n))=\RM(L_1).$$
Then $D+F=L_1$ by connectedness, and $L_1>F>L_3$. Clearly $L_1/F\cong D/(D\cap F)$ is divisible.\end{proof}

\begin{corollary}\label{c:bottom}
Let $B$ be a connected bisoluble skew brace of finite Morley rank. Then there is a left chief series $(L_i:i\le n)$ and $\ell\le n$ such that $L_i/L_{i+1}$ is divisible for $i<\ell$, and has finite exponent for $i\ge\ell$.
\end{corollary}
\begin{proof}
By Lemma \ref{l:finmin=minfin}, if a divisible left chief factor is just below one of finite exponent, we can exchange them. Iterating this process will yield a left chief series as required.\end{proof}

\begin{lemma}\label{cosesopraleft}
Let $B$ be a connected bisoluble skew brace of finite Morley rank. If $L$ is a maximal proper connected strong left ideal in $B$, then $C\circ L$ is a strong left ideal in $B$ for every normal multiplicative subgroup $C\le B_\circ'$. Therefore, if $C$ is connected, either $C\circ L=B$ or $C\subseteq L$.
\end{lemma}
\begin{proof}
$B/L$ is abelian and the $\lambda$-action of $B'_\circ$ on $B/L$ is trivial by Lemma \ref{Minimality}. Let $b\in B$, $c\in C$ and $\ell\in L$. Then $\ell':=\lambda_b(\ell)\in L$, so
$$\lambda_b(c+\ell)=\lambda_b(c)+ \ell'=-b+b\circ c+\ell'=-b+b\circ c\circ b^{-1}\circ b+\ell'.$$
But $b\circ c\circ b^{-1}\in C\le B'_{\circ}$ and $B'_{\circ}\ast B\subseteq L$, so $$\lambda_b(c+\ell)=-b+b\circ c\circ b^{-1}+b+\ell''=b\circ c\circ b^{-1}+\ell'''\in C+L$$ for some $\ell'',\ell'''\in L$. Thus, $C+L$ is $\lambda$-invariant; since $C+L=C\circ L$ is a multiplicative subgroup of $B$, it is also an additive subgroup of $B$, which must be normal as $B/L$ is abelian. Thus, $C+L$ is a strong left ideal of $B$.
\end{proof}

\begin{corollary}\label{SolBiNil}
Let $B$ be a connected skew brace of finite Morley rank such that $B_+$ is soluble and $B_\circ$ is nilpotent. If $L$ is a maximal connected strong left ideal of $B$, then $L$ is an ideal of $B$.
\end{corollary}
\begin{proof}
Recall that the multiplicative lower central series is definable and connected. Let $C$ be its last term such that $C\circ L=B$, and put $D=[C,B]_\circ$. Then $D\le B'_\circ$ is connected, as is $D\circ L$. So $D\circ L$ is a connected strong left ideal of $B$ by Lemma \ref{cosesopraleft}, proper in $B$ by our choice of $C$. Hence $D\circ L=L$ and $D\le L$. It follows that $L$ is multiplicatively normalized by $C$, whence by $B=C\circ L$. Thus $L$ is an ideal.
\end{proof}

\begin{lemma}\label{BLL_1}
Let $B$ be a connected skew brace of finite Morley rank, and $L/M$ a left chief factor of $B$.
\begin{enumerate}[(1)]
\item If $B$ is additively nilpotent, then $L/M\leq Z^+(B/M)$.
\item If $B$ is left nilpotent of nilpotent type, the $\lambda$-action of $B$ on $L/M$ is trivial.
\end{enumerate}  
\end{lemma}
\begin{proof}\begin{enumerate}[(1)]
\item By nilpotency of $B/M$ we have $$M/M\le ([B,L]_++M)/M=[B/M,L/M]_+<L/M.$$
Since $[B,L]_++M$ is a definable connected strong left ideal of $B$, it must equal $M$ by $\B$-minimality of $L/M$.
\item $\B$ is nilpotent by Theorem \ref{t:lnp-Bnp} and $\mathfrak M_0$ is a normal subgroup of $\B$. So $\B/\mathfrak M_0$ is nilpotent, and
$$\mathfrak M_0/\mathfrak M_0\le(([B,L]+M)\times\{0\})/\mathfrak M_0=[\mathfrak L_0/\mathfrak M_0,\B/\mathfrak M_0]<\mathfrak L_0/\mathfrak M_0.$$
As $[B,L]+M$ is definable connected strong left ideal in $B$ properly contained in $L$, we have $[B,L]\le M$. In particular $B\ast L\subseteq M$, so the $\lambda$-action on $L/M$ is trivial.\qedhere\end{enumerate}
\end{proof}

\begin{corollary}
Let $B$ be a connected left nilpotent skew brace of nilpotent type and of finite Morley rank. If  $L$ is a connected ideal of $B$ which is also a minimal strong left ideal, then $L\subseteq\ann(B)$.
\end{corollary}
\begin{proof}
We have $L\leq Z^+(B)$ and $B\ast L=\{0\}$ by Lemma \ref{BLL_1}. Thus $L$ is a trivial brace whose additive subgroups are strong left ideals in $B$. Therefore $L$ is a minimal normal multiplicative subgroup of $B$, and $L\leq Z^\circ(B)$ by Lemma \ref{center}. Thus $L\subseteq \ann(B)$. 
\end{proof}

\section{Bisoluble skew braces of finite Morley rank}\label{sec:4}
In this section, we prove that bisoluble connected skew braces of finite Morley rank and left chief length less than or equal to $3$ are weakly soluble. We start with an easy case.

\begin{proposition}\label{ChiefLen1}
Let $B$ be a bisoluble connected skew brace of finite Morley rank with left chief length $1$. Then $B$ is trivial, minimal and of abelian type.
\end{proposition}
\begin{proof}
By Lemma \ref{Minimality} the $\lambda$-action of $B_\circ'$ on $B$ is trivial and $B_+$ is abelian, so any proper left ideal is strong, whence finite. Then $\ker(\lambda)$ satisfies (i), (ii) and (iv) of Lemma \ref{lemuseful}, and must be an ideal. If $B_\circ'$ is non-trivial, it is infinite by connectivity, so $\ker(\lambda)$ is an infinite ideal and must equal $B$. But then $B$ is trivial, of abelian type, and minimal.

So we may assume that $B_\circ$ is abelian as well. If $B$ is non-trivial, there is $0\neq a\in B$ such that the map $\ast_a:x\mapsto a\ast x$ is a non-trivial homomorphism from $B_+$ to $B_+$ by Lemma \ref{WelDefstar}. Moreover, $\ast_a$ is not surjective, as otherwise there would exist $b\in B$ with $a\ast b=-a$, whence $\lambda_b(a)=0$ and $a=0$, a contradiction. Now $\ker(\ast_a)$ cannot be finite, as then $\im(\ast_a)=B$ by connectivity. Note that $\ker(\ast_a)=\{x\in B:\lambda_a(x)=x\}$, and
$$\lambda_a(\lambda_b(c))=\lambda_{b}(\lambda_a(c))=\lambda_b(c)$$
for every $b\in B$ and $c\in\ker(\ast_a)$. Hence $\ker(\ast_a)^0$ is an infinite proper connected ideal, contradicting the fact that $B$ has left chief length $1$.

It follows that $B$ is trivial, minimal, and of abelian type.
\end{proof}

\begin{corollary}\label{mrk1}
Let $B$ be a connected skew brace of Morley rank $1$. Then $B$ is trivial and of abelian type.
\end{corollary}

\begin{corollary}\label{c:mintriv}
Let $B$ be a bisoluble connected skew brace, and $L$ a minimal strong left ideal. Then $L$ is a trivial skew brace of abelian type.
\end{corollary}
\begin{proof}
Lemma \ref{Minimality} shows that $L$ is abelian, and either a skew brace of left chief length $1$, whence trivial by Proposition \ref{ChiefLen1}, or is acted upon trivially by $\mathfrak L$. In either case $L$ is a trivial skew brace of abelian type.
\end{proof}

\begin{corollary}\label{lemsolrk3}
Let $B$ be a connected skew brace of finite Morley rank. If $B$ is of soluble type and multiplicatively nilpotent, then $B$ is strongly left soluble. More precisely, any left chief series is abelian. In particular, if $B_\circ$ is abelian, then $B$ is soluble.
\end{corollary} 
\begin{proof} If $B$ has left chief length $1$, we are done by Proposition \ref{ChiefLen1}. So suppose otherwise, and let $(L_i:i\le n)$ be a $B$-left chief series. Then $L_1$ is an ideal by Corollary \ref{SolBiNil}; by induction on $\RM(B)$ any $L$-left chief series of $L$ is abelian.
	
Refine the series $(L_i:1\le i\le n)$ into an $L$-left chief series of $L$. By Lemma \ref{Minimality} any $B$-chief factor $L_i/L_{i+1}$ for $1\le i<n$ is either an $L$-chief factor and $\mathfrak L_i$ acts trivially on it by inductive hypothesis, or $\mathfrak L$ acts trivially on it. Moreover $L_0/L_1$ is a trivial brace by Proposition \ref{ChiefLen1}. It follows that $(L_i:i\le n)$ is abelian, and $B$ is strongly left soluble.

If $B_\circ$ is abelian, then strong left ideals are ideals, so $B$ is soluble.
\end{proof}
	
\begin{lemma}\label{BisFinCha}
Let $B$ be a connected bisoluble skew brace of finite Morley rank, and let $L$ be a maximal connected strong left ideal of $B$ such that either $(L,+)$ or $(L,\circ)$ is of finite exponent. Then $L$ is an ideal of $B$, both $(L,\circ)$ and $(L,+)$ are of finite exponent, and $B$ is strongly left soluble. More precisely, any left chief series passing through $L$ is abelian.
\end{lemma}
\begin{proof}Let $\mathcal L$ be a left chief series for $B$ passing through $L$. If $L$ has left chief length $1$, it is abelian and trivial by Proposition \ref{ChiefLen1}, and of finite exponent. Otherwise consider a left chief series for $L$ refining $\mathcal L\setminus\{B\}$. The maximal connected strong left ideal of $L$ in that series has finite (additive or multiplicative) exponent. By induction on Morley rank $L$ is strongly left soluble, so both $L_+$ and $L_\circ$ have finite exponent by Remark \ref{r:ws-fe}.
	
In either case the $\lambda$-action of $L$ on $B/L$ is trivial since otherwise by connectedness $L_\circ/\ker(\lambda^{B/L})$ would be isogenous to an infinite subgroup of the multiplicative group of an algebraically closed field, contradicting Lemma \ref{Kequi}. Therefore $L\ast B\subseteq L$, and $L$ is an ideal by Lemma \ref{lemuseful}. We now finish using Lemma \ref{Minimality} as in the proof of Corollary \ref{lemsolrk3}.
\end{proof}

Next, we deal with the case of left chief length $2$.

\begin{proposition}\label{GenCas1}
Let $B$ be a connected bisoluble skew brace of finite Morley rank and left chief length at most $2$. Then $B$ is soluble of derived length at most $2$.
\end{proposition}
\begin{proof}
If $B$ has left chief length $1$, it is trivial, minimal and abelian by Proposition \ref{ChiefLen1}. So we may assume that B has a proper infinite connected strong left ideal $L$. By Lemma \ref{Minimality} both $L$ and $B/L$ are additively abelian, and either minimal or of finite exponent. Moreover, $L$ is a trivial skew brace by Corollary \ref{c:mintriv}.

If $L$ is of finite exponent, it is an ideal by Lemma \ref{BisFinCha}, so $B/L$ is a trivial skew brace and $B$ is soluble. Therefore, we may assume that $L$ is not of finite exponent, whence minimal; by Lemma \ref{l:finmin=minfin} we may assume that $B/L$ is minimal as well. Now by Lemma \ref{cosesopraleft} either  $B_\circ'\circ L=B$ or $B_\circ'\subseteq L$. In the second case $L$ is an ideal, so $B/L$ is a trivial brace by Proposition \ref{ChiefLen1} and $B$ is soluble. Similarly, if $L\ast B\subseteq L$ then $L$ is an ideal by Lemma \ref{lemuseful}, and $B$ is again soluble.

Finally, assume that $B_\circ'\circ L=B$ and the $\lambda$-action of $L$ on $B/L$ is non-trivial. So $B/L\cong K_+$ for some algebraically closed field $K$, and $L_\circ$ embeds isogenically into $K^\times$ (the kernel must be finite by minimality). Then $\RM(B/L)=\RM(K)\ge\RM(L)$.

Consider the additive conjugation action of $B$ on $L$. If it is non-trivial, then 
$C_B^+(L)<B$ is a strong left ideal in $B$ containing $L$ so $C_B^+(L)^0=L$ by minimality of $B/L$. The action is irreducible by minimality of $L$, so $L_+=L_\circ$ is isomorphic to the additive group of an algebraically closed field, while $B/L$ isogenically embeds into its multiplicative group. Then $\RM(B/L)\le\RM(L)$, and equality must hold. It follows that $L\isog K^\times$, contradicting Lemma \ref{Kequi}. Hence $C_B^+(L)=B$ by $\B$-minimality of $B/L$, i.e.\ $L\le Z^+(B)$, and $B_+$ is abelian by Lemma \ref{lemnotbothfe}.

Now, let $b\in B_\circ'$. Since the $\lambda$-action of $B'_{\circ}$ on $L_+$ and on $B/L$ is trivial, we have $B_\circ'\ast L=\{0\}$ and $B_\circ'\ast B\subseteq L$, so the map  
$$\ast_b:x+L\mapsto b\ast x$$
from $B/L$ to $L$ is a well-defined homomorphism by Lemma \ref{WelDefstar}. If $\ast_b$ is not zero, then it is an isogeny by minimality of $B/L$ and $L$, so $K_+\cong B/L\isog L\isog K^\times$, contradicting Lemma \ref{Kequi}.

It follows that $B_\circ'\ast B=\{0\}$. So addition equals multiplication on $B_\circ'$ which is a multiplicatively normal skew sub-brace; it is trivially additively normal, and an ideal by Lemma \ref{lemuseful}. Then $B/B_\circ'$ and $B_\circ'$ are both minimal trivial braces by Proposition \ref{ChiefLen1}, and $B$ is soluble of derived length at most $2$.
\end{proof}

\begin{corollary}\label{rank2soluble}
Let $B$ be a connected skew brace of Morley rank $2$. Then $B$ is soluble of derived length at most $2$.
\end{corollary}
\begin{proof}
Both $B_+$ and $B_\circ$ are soluble by Theorem \ref{structuremorleyrank2}. The result now follows from Proposition~\ref{GenCas1}.
\end{proof}

\begin{lemma}\label{NoIdeLen2}
Let $B$ be a connected bisoluble skew brace of finite Morley rank. Let $L$ be a connected strong left ideal of $B$-left chief length at most $2$. Then $L$ is strongly left soluble.
\end{lemma}
\begin{proof}
If the $B$-left chief length of $L$ is $1$, then $L$ is trivial of abelian type by Corollary \ref{c:mintriv}.

We now assume that the $B$-left chief length of $L$ is $2$, and consider a proper infinite connected strong $B$-left ideal $A$ of $L$. Then, by Lemma \ref{Minimality}, both $A_+$ and $L/A$ are minimal or of finite exponent. If both $A_+$ and $L/A$ are minimal, then $L$ is soluble by Proposition \ref{GenCas1}. Otherwise we may assume $A$ is of finite exponent by Lemma \ref{l:finmin=minfin}, and $L/A$ is either minimal or of finite exponent. Then $L$ is strongly left soluble by Lemma \ref{BisFinCha}.
\end{proof}

We can now prove the Main Theorem of this section.
\begin{theorem}\label{SolLA}
Let $B$ be a connected bisoluble skew brace of finite Morley rank and left chief length $3$. Then $B$ is weakly soluble.
\end{theorem}
\begin{proof}
Assume for a contradiction that $B$ is not weakly soluble.

\begin{claim}\label{claimuno}
$B$ has no definable infinite proper ideals.
\end{claim}
\begin{claimproof}
Suppose, for a contradiction, that $B$ has an infinite definable proper ideal $I$. Replacing $I$ by its connected component, we may assume that $I$ is connected. As $I$ is proper and connected, $I$ has $B$-left chief length at most $2$. Moreover, $B/I$ has left chief length at most $2$ because $I$ is infinite. Then $B/I$ is soluble by Proposition \ref{GenCas1}, and $I$ is strongly left soluble by Lemma \ref{NoIdeLen2}.  Therefore, $B$ is weakly soluble, a contradiction.
\end{claimproof}  

Let $B>L>A>\{0\}$ be a left chief series of $B$. Then $B/L$, $L/A$, and $A_+$ are abelian by Lemma \ref{Minimality}, and either minimal or of finite exponent. Moreover, $A$ is trivial by Corollary \ref{c:mintriv}. The $\lambda$-action of $B$ on $B/L$ is non-trivial, as otherwise $L$ would be an ideal of $B$, contradicting Claim~\ref{claimuno}. Hence there exists an algebraically closed field $K$ such that $B/L\cong K_+$. Note that the $\lambda$-action of $B'_{\circ}$ on $B/L$ is trivial by Lemma \ref{TriActG'}.

If $B_\circ'\le L_\circ$, then $L$ is an ideal by Lemma \ref{cosesopraleft}, contradicting Claim~\ref{claimuno}. Let $C$ be a minimal definable normal connected subgroup of $B_\circ$ contained in $B_\circ'$ but not in $L_\circ$. Then $C\circ L_\circ=B$ by Lemma \ref{cosesopraleft}; moreover $B_\circ'$ is nilpotent by Lemma \ref{NilpDerG}, so $[C,B_\circ']_\circ$ is proper in $C$, definable, connected and normal in $B_\circ$, whence $[C,B_\circ']_\circ\le L_\circ$ by minimality. In particular, $C\cap L_\circ$ contains $C_\circ'$ and is normalized by $L_\circ$ and by $C$, whence by $B_\circ$. Note that $C\not\le L_\circ$ implies that the multiplicative quotient $C/(C\cap L_\circ)$ is infinite by connectivity of $C$, and irreducible for the conjugation action of $B_\circ$ by our minimal choice of $C$.

Consider the multiplicative groups $S=C_B^\circ(C/(C\cap L_\circ))$ and $D=(S\cap L_\circ)^0$. Then $[C,B_\circ']_\circ\le C\cap L_\circ$ implies $C\le B_\circ'\le S$, so $S\cap L_\circ$ is normalized by $L_\circ$, but also by $C$ because 
$$[C,(S\cap L_\circ)]_\circ\leq C\cap L_\circ\leq S\cap L_\circ.$$
Thus, $S\cap L$ is multiplicatively normal in $B=C\circ L$, as is its connected component $D$.
Moreover,
$$B_\circ/S=(S\circ L_\circ)/S\cong L_\circ/(S\cap L_\circ),$$
and $L_\circ/D$ is isogenous to $B_\circ/S$.

If $S=B$ then $[B,C]_{\circ}\le L_\circ$, so $L_\circ$ is normalized by $C$, whence by $C\circ L_\circ=B_\circ$, and $L$ is an ideal, contradicting Claim~\ref{claimuno}. Thus, $B'_\circ\le S<B_\circ$ and $B_\circ/S$ is infinite abelian by connectedness of $B$. Then $B_\circ/S$ acts faithfully by multiplicative conjugation on $C/(C\cap L_\circ)$; the action is irreducible by $B_\circ$-minimality of $C/(C\cap L_\circ)$. By Theorem \ref{ZilberLin} there is an algebraically closed field $F$ such that $B_\circ/S$ embeds into $F^\times$ and 
$C/(C\cap L_\circ)\cong F_+$. Then $L_\circ/D$ cannot have infinite subgroups of finite exponent by Lemma \ref{Kequi}. Moreover,
$$\begin{aligned}\RM(L_\circ/D)\leq\RM(F)&=\RM(C/(C\cap L_\circ))=\RM((C\circ L_\circ)/L_\circ)\\
&=\RM(B_\circ/L_\circ)=\RM(B/L)=\RM(K),\end{aligned}$$
as additive and multiplicative cosets modulo a left ideal coincide.
\begin{claim}\label{claimD}$D\not=L$ and $D\not=A$. If $\RM(B/L)\le\RM(L/A)$, then $D$ is infinite.\end{claim}
\begin{claimproof}
If $D=L$ then $L$ is multiplicatively normal, whence an ideal. Similarly, if $D=A$ then $A$ is an ideal; either case contradicts Claim~\ref{claimuno}.

For the second assertion, if $D$ is finite and $\RM(B/L)\le\RM(L/A)$, then
$$\RM(L)=\RM(L_\circ/D)\le\RM(B/L)\le\RM(L/A)=\RM(L)-\RM(A)<\RM(L)$$
since $A$ is infinite, a contradiction.
\end{claimproof}

\begin{claim}%\label{claimdue}
All chief factors are minimal divisible. Moreover, $L$ is soluble of derived length at most $2$.
\end{claim}
\begin{claimproof}
By Corollary \ref{c:bottom}, if there is at least one chief factor of finite exponent, we may assume that it is $A_+$; if there is a second one, we may assume it is $L/A$. In that case $L_+$ is of finite exponent, so $B$ is strongly left soluble by Lemma \ref{BisFinCha}, a contradiction. Therefore $B/L$ and $L/A$ must be minimal divisible.

Now, $A$ is a maximal connected strong left ideal of $L$ of additive finite exponent, so $L$ is weakly soluble and $A$ is a (trivial) ideal in $L$ of finite exponent by Lemma \ref{BisFinCha}. Therefore, $L/A$ is a trivial brace with minimal additive (and so also multiplicative) group by Proposition \ref{ChiefLen1}. Hence, $D\circ A/A$ equals either $L/A$ or $A/A$. In the first case, $L_\circ/D$ has finite exponent, and must be finite as $L_\circ/D$ has no infinite subgroups of finite exponent. By connectedness $D=L$, contradicting Claim \ref{claimD}. Therefore $D\subseteq A$, whence even $D=A$ because $A_\circ$ has finite exponent and $A_\circ/D\le L_\circ/D$, again contradicting Claim~\ref{claimD}.

It follows that all chief factors are minimal divisible. In particular $L$ has left chief length $2$ and is soluble  of derived length at most $2$ by Proposition \ref{GenCas1}.
\end{claimproof}

Consider the connected ideal $[L,L]$ of $L$. If $[L,L]=\{0\}$ then $L$ is trivial and we put $I=A$; otherwise put $I=[L,L]$.
\begin{claim}\label{c:I}
$I$ is an infinite connected proper ideal in $L$ such that $L/I$ and $I$ are trivial, minimal and of infinite exponent. Either $I=A$ or $I\cap A$ is finite; in the latter case $L=I+A=I\circ A$. If $D\not=\{0\}$ then $D$ and $L_\circ/D$ are minimal and abelian.
\end{claim}
\begin{claimproof}
If $L$ is trivial, a strong left ideal is an ideal. Otherwise $I=[L,L]$ is an ideal in $L$ by definition. The series $L>A>\{0\}$ is a $\{1\}$-composition series of $L_0$; by Lemma \ref{Jordan-Holder} the series $L>I>\{0\}$ is an equivalent one, so $L/I$ and $I$ are also minimal and of infinite exponent. If $I\not=A$ then $I\cap A$ is finite by minimality of $A$, so $I+A=L$ by minimality of $L/A$.

By triviality and minimality of $L/I$ and $I$, the series $L_\circ>I_\circ>\{0\}$ is a multiplicative $\{0\}$-composition series for $L_\circ$ with (multiplicatively) minimal factors. But $D<L_\circ$ by Claim \ref{claimD}, so if $D\not=\{0\}$, the series $L_\circ>D>\{0\}$ is also a multiplicative $\{0\}$-composition series. Then $L_\circ/D$ and $D$ are minimal by Lemma \ref{Jordan-Holder}, whence abelian.
\end{claimproof}

\begin{claim}\label{claimduebis}
$D\cap A$ is finite, $A$ isogenically embeds into $B/S$, and $\RM(A)\leq \RM(B/L)$.
\end{claim}
\begin{claimproof} If $D=\{0\}$, clearly $D\cap A$ is finite. So suppose $D\not=\{0\}$. Then
$L_\circ/D$ and $D$ are both minimal; $D\not=A$ now implies $D\cap A$ finite.
Hence $A$ embeds isogenically into $L_\circ/D$. In particular, $A$ isogenically embeds into $F^\times$, whence $\RM(A)\leq\RM(F)=\RM(B/L)$. 
\end{claimproof}

\begin{claim}\label{c:isog}
$B/L$ is not definably isogenous to $A$.
\end{claim}
\begin{claimproof}
Otherwise $\RM(F)=\RM(K)=\RM(B/L)=\RM(A)$, so the isogenous embedding of $A$ into $F^\times$ is surjective. Then $F^\times\isog A\isog B/L\isog K_+$, contradicting Lemma \ref{Kequi}.
\end{claimproof}

\begin{claim}\label{claimtre}
$A\le Z^+(B)$ and $L_+$ is abelian.
\end{claim}
\begin{claimproof}
Consider the additive conjugation action of $B$ on $A$, which is irreducible by minimality of $A$.
Put $M=C_B^+(A)^0$, a connected strong left ideal in $B$ containing $A$. If $M<B$, the quotient $B/M$ isogenically embeds into the multiplicative group $\ell^\times$ of an algebraically closed field $\ell$, while $A_+$ is isomorphic to $\ell_+$. Thus 
$$\RM(B/M)\le\RM(\ell)=\RM(A)\le\RM(B/L),$$
whence  $\RM(M)\geq\RM(L)>\RM(A)$. Therefore $B>M>A>\{0\}$ is a left chief series of $B$. 

By the proof up to Claim \ref{claimduebis} with $L$ replaced by $M$, we see that $B/M$ is isomorphic to the additive group $k_+$ of an algebraically closed field, and $A$ isogenically embeds into $k^\times$. In particular $\RM(A)\leq\RM(B/M)$, so we have equality and $\RM(B/M)=\RM(k)$. It follows that $B/M\cong k_+$ is isogenous to $\ell^\times$, contradicting Lemma~\ref{Kequi}.

The last assertion follows from Lemma \ref{lemnotbothfe} since $A\le Z^+(L)$.
\end{claimproof}

\begin{claim}\label{c:X}
If $X$ is an infinite connected proper additive subgroup of $L$, then either $X$ is $\B$-invariant or $L/A\isog A\isog X$.
\end{claim}
\begin{claimproof}
Suppose $X$ is not $\B$-invariant. Then $X$ has an automorphic image $X'\not=X$. By Lemma \ref{Jordan-Holder} all of $L/X$, $X$, $L/X'$ and $X'$ are minimal, since $L/A$ and $A$ are. Then 
$X=A$ or $X\cap A$ is finite by minimality of $A$; similarly $X'=A$ or $X'\cap A$ is finite; as $X\not=X'$ we may assume $X\cap A$ finite. Moreover $X\cap X'$ is finite by minimality of $X$, and $X+A=X+X'=L$ by minimality of $L/X$. Then 
$L/A\isog X\cong X'\isog L/X\isog A$.
\end{claimproof}

\begin{claim}\label{claimquatro}
$\RM(L/A)\leq \RM(B/L)$.
\end{claim}
\begin{claimproof}
If $I$ is not $\B$-invariant, then $L/A\isog A$ by Claim \ref{c:X}, so $\RM(L/A)=\RM(A)\le\RM(B/L)$. Otherwise $I$ is a connected $B$-left ideal, so $B>L>I>\{0\}$ is a left chief series for $B$; replacing $A$ by $I$ in the proof up to Claim~\ref{claimduebis}, we obtain $\RM(I)\le\RM(B/L)$. So if $I\cap A$ is finite, we have $L=I+A$ and $\RM(L/A)=\RM(I)\le\RM(B/L)$.

So we may assume $I=A$ is an ideal in $L$. Assume for a contradiction that $\RM(L/A)>\RM(B/L)$, and recall $\RM(B/L)\ge\RM(A)$ by Claim \ref{claimduebis}.

Suppose first that $\lambda$-action of $L$ on $A$ is non-trivial. Then
$L_\circ>\ker(\lambda^A)\ge A_\circ$ by triviality of $A$. But $L/A$ is trivial and minimal, so 
$\ker(\lambda^A)^0=A$ and $L/A$ isogenically embeds into the multiplicative group of a field whose additive group is isomorphic to $A$, contradicting $\RM(A)\le\RM(B/L)<\RM(L/A)$. It follows that $L\ast A=\{0\}$, and $L\ast L\subseteq I=A$.

Since $L_+$ is abelian by Claim \ref{claimtre}, we next consider $\ast_\ell:x+A\mapsto \ell\ast x$ for $\ell\in L$, a well-defined homomorphism from $L/A$ to $A$ by Lemma \ref{WelDefstar}. But $RM(L/A)>\RM(A)$ and $L/A$ is minimal, so $\ker(\ast_\ell)=L/A$ and $L$ must be trivial.

Now $D$ is infinite by Claim~\ref{claimD}, whence a minimal sub-brace by Claim \ref{c:I} and triviality of $L$. If $D$ were $\B$-invariant, it would be an ideal, contradicting Claim \ref{claimuno}. So $L/A\isog A$ by Claim \ref{c:X}, and $\RM(L/A)=\RM(A)\le\RM(B/L)$.
\end{claimproof}

\begin{claim}\label{claimcinque}
$B_+$ is abelian.
\end{claim}
\begin{claimproof}
Let us show first that $B/A$ is abelian. Consider the additive conjugation action of $B$ on $L/A$. It is irreducible by minimality of $L/A$.
Now $L\le C_B^+(L/A)$ as $L/A$ is abelian; if $C_B^+(L/A)<B$, then $C_B^+(L/A)^0=L$ by minimality of $B/L$, so $B/L$ embeds isogenically into the multiplicative group $k^\times$ of an algebraically closed field $k$ with $k_+\cong L/A$. But $\RM(k)=\RM(L/A)\le\RM(B/L)$ by Claim~\ref{claimquatro}, so $B/L$ and $k^\times$ are isogenous. On the other hand, $B/L$ is isogenous to $K_+$, contradicting Lemma \ref{Kequi}. It follows that $L/A\le Z^+(B/A)$, so $B/A$ is abelian by Lemma \ref{lemnotbothfe}.

Now $A\le Z^+(B)$ and both $B/A$ and $L_+$ are abelian by the first paragraph and Claim~\ref{claimtre}, so for any $\ell\in L$, the map $x+L\mapsto [\ell,x]_+$ is an additive homomorphism from $B/L$ to $A$. If it is non-trivial, it is an isogeny by minimality, contradicting Claim \ref{c:isog}. Hence $L\le Z^+(B)$.
	
Finally, for $b\in B$ the map $x+L\mapsto [b,x]_+$ is an additive homomorphism from $B/L$ to $A$, which cannot be an isogeny and must be trivial again. Thus, $B_+$ is abelian.
\end{claimproof}

\begin{claim}\label{c:DLab}
	If $D$ is infinite, $L_\circ$ is abelian.
\end{claim}
\begin{claimproof}
By minimality either $D=I$ or $D\cap I$ is finite. In the first case, since $I$ cannot be an ideal by Claim \ref{claimuno}, we have $L/A\isog A \isog I$ by Claim \ref{c:X}. Consider the multiplicative conjugation action of $L_\circ$ on $D$. It is irreducible by minimality, so if it is non-trivial, $D=I_\circ=I_+\isog A_+$ is isomorphic to the additive group of an algebraically closed field. The kernel of the action must be a finite extension of $D$ by minimality of $L_\circ/D$, so $A_+=A_\circ\isog L_\circ/D$ embeds isogenically into the multiplicative group of the field; as $\RM(A)=\RM(I)$ the isogenic embedding is surjective, contradicting Lemma \ref{Kequi}. It follows that $D\le Z^\circ(L)$, so $L_\circ$ is abelian by Lemma \ref{lemnotbothfe}.
	
On the other hand, if $D\cap I$ is finite, $L_\circ$ is the product of two connected abelian normal subgroups with finite intersection. Then $[D,I]\le D\cap I$ is connected and finite, whence trivial. Thus, in either case $L_\circ$ is abelian.
\end{claimproof}

\begin{claim}\label{claimsei}
$B'_{\circ}\ast B\subseteq A$.
\end{claim}
\begin{claimproof}$B_+$ is abelian by Claim \ref{claimcinque}; moreover $B_\circ'\ast B\subseteq L$ and $B_\circ'\ast L\subseteq A$ as $B_\circ'$ acts trivially on chief factors by Lemma \ref{Minimality}. Hence for any $b\in B_\circ'$ the map
$$\ast_b:x+L\mapsto b\ast x+A$$
is a homomorphism from $B/L$ to $L/A$ by Lemma \ref{WelDefstar}. Suppose for a contradiction that it is non-trivial. By minimality $\ast_b$ is an isogeny and $\RM(B/L)=\RM(L/A)$, so $D$ is infinite and minimal by Claims \ref{claimD} and \ref{c:I}. Then $L_\circ$ is abelian by Claim \ref{c:DLab}.

Since $L_+$ is abelian and $I\le L$ is a trivial ideal with $L\ast L\subseteq I$, the map $\ast_i:x+I\mapsto i\ast x$ is a well-defined homomorphism from $L/I$ to $I$ for every $i\in I$. Suppose $i\in I$ is such that $\ast_i$ is non-trivial. By minimality it is an isogeny, so $I\isog L/I$. By Lemma \ref{Jordan-Holder} this implies $A\isog L/A\isog B/L$, contradicting Claim \ref{c:isog}. Hence, $I\ast L=\{0\}$; by commutativity of $L_+$ and $L_\circ$ this implies $L\ast I=\{0\}$.

But now for $\ell\in L$ the map $\ast_\ell:x+I\mapsto \ell\ast x$ is a well-defined homomorphism from $L/I$ to $I$, which must be trivial for the same reason as $\ast_i$. It follows that $L$ is a trivial brace, and $D$ is a sub-brace; it is trivial, and minimal by Claim \ref{c:I}. But then $L/A\isog A$ by Claims \ref{c:X} and \ref{claimuno}, and $A\isog B/L$, again contradicting Claim \ref{c:isog}.

This contradiction shows that all the homomorphisms $\ast_b$ for $b\in B_\circ'$ are trivial, so $B'_{\circ}\ast B\subseteq A$.
\end{claimproof}

\begin{claim}\label{claimsette}
$B'_{\circ}\circ A=B$, $B\ast A=\{0\}$, and $D=(B_\circ'\cap L_\circ)^0$ is infinite.
\end{claim}
\begin{claimproof}Recall that $B_+$ is abelian by Claim~\ref{claimcinque}.
Let $b,b'\in B_\circ'$. As $-b+b\circ b'-b'=b\ast b'\in A$ by Claim \ref{claimsei}, we have
$$b\circ b'\in b+b'+A,$$
whence $b+b'\in B_\circ'+A$. Hence $B_\circ'\circ A=B_\circ'+A$ is a sub-brace of $B$. Now, if $b'\in B'_{\circ}$ and $b\in B$, then by multiplicative normality there is $b''\in B_\circ'$ with $b\circ b'=b''\circ b$, so
$$\lambda_b(b')=-b+b\circ b'=-b+b''\circ b=-b+b''+b''\ast b+b\in B_\circ'+A,$$
since $b''\ast b\in A$. It follows that $B_\circ'+A$ is a strong left ideal in $B$. But it contains $B_\circ'$ and must be multiplicatively normal, i.e.\ an ideal. So $B_\circ'+A=B_\circ'\circ A=B$ by Claim~\ref{claimuno}.

As the $\lambda$-actions of $B_\circ'$ and of $A$ on $A$ are both trivial, $B\ast A=\{0\}$.

Finally, $B=B_\circ'\circ A$ implies $L=(B'_\circ\cap L_\circ)^0\circ A$. Since $B_\circ'\le S$, we have $(B'_\circ\cap L_\circ)^0\le (S\cap L_\circ)^0=D$, and we must have equality because $A\cap D$ is finite. So
$$\RM(D)=\RM(B'_\circ\cap L_\circ)=\RM(L)-\RM(A)>0$$
and $D$ is infinite.
\end{claimproof}

Since $D$ is infinite, $L_\circ$ is abelian by Claim \ref{c:DLab}; moreover $D\le Z^\circ(B'_\circ)$ by Lemma \ref{center}; as $B=B_\circ'\circ A$, we even have $D\leq Z^\circ(B)$.

Now $D\ast A=\{0\}$ by Claim~\ref{claimsette} and $D\ast B\subseteq A$ by Claim~\ref{claimsei}; as $B_+$ is abelian by Claim~\ref{claimcinque}, for $d\in D$ the map
$$\ast_d:x+A\mapsto d\ast x$$
is a well-defined homomorphism from $B/A$ to $A$. Suppose it were non-trivial, and consider $b\in B$ and $x\in\ker(\ast_d)$. Then
$$\lambda_d\lambda_b(x)=\lambda_b\lambda_d(x)=\lambda_b(x),$$
so $\ker(\ast_d)$ is $\lambda$-invariant, whence a strong left ideal. Put $U=\ker(\ast_d)^0$, a proper connected strong left ideal of $B$. Since
$$\RM(U)\geq \RM(B)-\RM(A)>\RM(B)-\RM(L)=\RM(B/L)\geq \RM(A),$$
$U$ properly contains $A$. Then $B>U>A>\{0\}$ is a connected left chief series of $B$ such that $B/U$ is isogenous to $A$. Replacing $L$ by $U$, this contradicts Claim \ref{c:isog}.

It follows that $D\ast B=\{0\}$, so $D$ is a trivial sub-brace.   
Since $D$ is multiplicatively central and $B_+$ is abelian, $D$ is a proper infinite ideal by Lemma \ref{lemuseful}, and we have reached the final contradiction.
\end{proof}

\begin{corollary}
Let $B$ be a connected skew brace of Morley rank at most $3$. Then $B$ is weakly soluble if and only if $B$ is bisoluble.
\end{corollary}
\begin{proof}
This follows from Corollary \ref{mrk1}, Corollary \ref{rank2soluble}, and Theorem \ref{SolLA}.
\end{proof}

\section{Binilpotent skew braces of finite Morley rank}\label{sec:5}

In this section we deal with binilpotent skew braces of finite Morley rank. We shall first see that they decompose into an almost direct product of bi-$p$ ideals of finite exponent and radicable ideals.
\begin{theorem}\label{t:decomposition} Let $B$ be a bi-nilpotent connected skew brace of finite Morley rank. Then there are finitely many ideals $N_p$ for different primes $p$, and an ideal $D$, such that:\begin{itemize}
\item Every $N_p$ is bi-$p$ of finite exponent,
\item $D$ is radicable,
\item the sum of the $N_p$ is direct, and
\item $D\cap\bigoplus_p N_p$ is finite.
\end{itemize}\end{theorem}
\begin{proof}
Since $B_+$ is connected nilpotent, by \cite[Theorem 6.8]{borovik} it is the central sum of a characteristic connected group $N$ of bounded exponent and a characteristic connected radicable group $D$, such that $N\cap D$ is finite. Moreover, $N$ is a direct sum of finitely many connected characteristic $p$-groups $N_p$ of bounded exponent. As $D$ and all the $N_p$ are characteristic, they are strong left ideals in $B$. Then $B=D+\bigoplus_p N_p$; moreover $(N_p)_\circ$ is also a $p$-group of finite exponent for all $p$ by Remark \ref{r:ws-fe}.
	
Similarly, $B_\circ$ is connected and nilpotent, and the central product of characteristic connected groups $M$ and $C$, where $M$ has finite exponent and $C$ is radicable, and $M=\circledcirc_p M_p$ is a direct product of $p$-groups $M_p$ of bounded exponent, for certain $p$ (a priori not necessarily the same as in the previous paragraph). Then $B_\circ/M_p$ has only finite $p$-subgroups of bounded exponent, whence $N_p\le M_p$ by connectedness.
	
Put $D_p=D+\bigoplus_{q\not=p}N_q=D\circ\circledcirc_{q\not=p}N_q$, again a strong left ideal with $B=D_p+N_p=D_p\circ N_p$. Then $D_p$ is strongly left soluble by Corollary \ref{lemsolrk3}, and has a finite series $D_p=I_0>I_1>\cdots>I_k=\{0\}$ such that $I_{i+1}$ is an ideal in $I_i$ for all $i<k$ and the quotients $I_i/I_{i+1}$ are abelian trivial. Note that a Sylow $p$-subgroup $P$ of $D_p$ is a Sylow $p$-subgroup of $D$, and hence equals a finite extension of a finite sum of Pr\"ufer $p$-groups by Theorem \ref{sylow}. Hence any section of $P$ of bounded exponent must be finite. By Lemma \ref{SylowLifting} any additive $p$-subgroup of $I_i/I_{i+1}$ of bounded exponent is finite; by triviality of $I_i/I_{i+1}$ any multiplicative $p$-subgroup of $(I_i/I_{i+1})_\circ$ of bounded exponent is finite. Therefore $\big((M_p\cap I_i)/(M_p\cap I_{i+1})\big)_\circ$
must be finite for all $i<k$, so $M_p\cap D_p$ is finite. As $B=D_p\circ N_p$ and $N_p\le M_p$, we have $N_p=M_p$ by connectedness, so this is an ideal in $B$, for all~$p$. In particular $N=M$ is an ideal in $B$.
	
Now
$$D_\circ/(N_\circ\cap D_\circ)\cong (D_\circ \circ N_\circ)/N_\circ=B_\circ/N_\circ=(C\circ M)/M\cong C/(M\cap C).$$
So $D_\circ$ is isogenous to $C$ and cannot have an infinite subgroup of finite exponent.
Then $D_\circ/(D_\circ\cap C)\cong (D_\circ \circ C)/C\le B_\circ/C$ has finite exponent and must be finite, whence $D_\circ\le C$ by connectedness. Since $D\cap N=D\cap M$ and $C\cap N$ are both finite,
$$\RM(D)=\RM(B/N)=\RM(B)-\RM(N)=\RM(B_\circ)-\RM(M)=\RM(B_\circ/M)=\RM(C),$$
we have $D=C$ by connectedness. Thus $D$ is an ideal as well.
\end{proof}

\begin{definition}
Let $G$ be a connected $\Sigma$-group of finite Morley rank, and $X$ a connected definable group. $G$ is called {\em $X$-$\Sigma$-homogeneous} if it has a $\Sigma$-composition series whose quotients are all definably isogenous to $X$. It is $\Sigma$-homogeneous if it is $X$-$\Sigma$-homogeneous for some $X$.
The $X$-$\Sigma$-homogeneous component is the maximal connected $X$-$\Sigma$-homogeneous normal subgroup of $G$.
\end{definition}

This definition is reminiscent of Frécon's definition of homogenous subgroups and components \cite{freconhomogeneous}, but we have not worked out the precise connection.

\begin{remark}\label{r:series}\begin{itemize}
\item As any two $\Sigma$-composition series of $G$ are equivalent, $X$-$\Sigma$-homogeneity does not depend on the choice of $\Sigma$-composition series.
\item Any connected subnormal $\Sigma$-subgroup $H$ of a connected $X$-$\Sigma$-homogeneous $\Sigma$-group $G$ is still $X$-$\Sigma$-homogeneous, since $H$ is a term in some $\Sigma$-composition series of $G$.
\item If $H$ and $H^*$ are mutually normalizing $X$-$\Sigma$-homogeneous subgroups with $\Sigma$-composition series 
$$\begin{aligned}H&=H_0>H_1>\cdots>H_m=\{0\}\quad\mbox{and}\\ H^*&=H^*_0>H^*_1>\cdots>H^*_n=\{0\},\end{aligned}$$
respectively, then
$$HH^*=HH^*_0\ge HH^*_1\ge\cdots\ge HH^*_n=H_0>H_1>\cdots>H_m=\{0\}$$
is a $\Sigma$-composition series for $HH^*$ (possibly with repetitions), since by $\Sigma$-simplicity $(H\cap H^*_i)H^*_{i+1}$ is either equal to $H^*_i$ or a finite extension of $H^*_{i+1}$, which means that the quotient $HH^*_i/HH^*_{i+1}\cong H^*_i/(H\cap H^*_i)H^*_{i+1}$ is either trivial or isogenous to $H^*_i/H^*_{i+1}\isog X$. It follows that $HH^*$ is still $X$-$\Sigma$-homogeneous, and an $X$-$\Sigma$-homogeneous component exists.
\item If $X$ and $X^*$ are not isogenous, then an $X$-$\Sigma$-homogeneous component $H$ and an $X^*$-$\Sigma$-homogeneous component $H^*$ have finite intersection, as $(H\cap H^*)^0$ is a connected $\Sigma$-subgroup which is both $X$- and $X^*$-$\Sigma$-homogeneous.
\end{itemize}\end{remark}

\begin{lemma}\label{l:Xhom}
Let $G$ be a connected nilpotent group of finite Morley rank, and $H\le G$ a connected  $X$-$\Sigma$-homogeneous subgroup for some definable connected group $X$. Suppose $\Sigma$ is closed under conjugation by inner automorphisms. Then $\prod_{g\in G}H^g$ is definable, normal in $G$, and $X$-$\Sigma$-homogeneous.\end{lemma}
\begin{proof}
We shall show first that $K:=\prod_{g\in N_G(N_G(H))}H^g$ is again a $X$-$\Sigma$-homogeneous subgroup.
As $\Sigma$ is closed under conjugation by inner automorphisms, all conjugates $H^g$ for $g\in N_G(N_G(H))$ are still $\Sigma$-subgroups, and have the same normalizer $N_G(H)$. In particular they normalize one another, so all finite products $HH^{g_1}\cdots H^{g_n}$ are definable connected $\Sigma$-subgroups; by the finiteness of Morley rank and connectedness there is a maximal such product, which must equal $K$. But a finite product of pairwise normalizing $X$-$\Sigma$-homogeneous subgroups is itself $X$-$\Sigma$-homogeneous by Remark \ref{r:series}.

Now $N_G(K)\ge N_G(N_G(H))$, and unless $H$ is already normal, $|N_G(N_G(H)):N_G(H)|$ is infinite by the normalizer condition, Lemma \ref{l:NC}. Hence $\RM(N_G(K))>\RM(N_G(H))$ and we iterate the construction; by finiteness of $\RM(G)$ we must stop. But at the last iteration we have obtained an $X$-$\Sigma$-homogenous subgroup which is a product of conjugates of $H$ and normal in $G$; by normality it contains all conjugates $H^g$ for $g\in G$, and hence equals $\prod_{g\in G}H^g$.
\end{proof}

We now consider homogeneity in the context of skew braces.
\begin{definition}
Let $B$ be a connected skew brace of finite Morley rank, and $L$ a connected strong left ideal of $B$. If $X$ is a definable connected group, $L$ is called {\em $X$-homogeneous} if it is $X$-$\B$-homogeneous, and {\em homogeneous} if it is $X$-homogeneous for some $X$. The {\em $X$-homogeneous component} of $B$ is the $X$-$\B$-homogeneous component of $B_+$, i.e.\ the maximal connected $X$-homogeneous strong $B$-left ideal.
\end{definition}

\begin{proposition}\label{p:homogeneous}
Let $B$ be a connected radicable binilpotent skew brace of finite Morley rank, and $A$ a homogeneous component of $B$. Then $A$ is an ideal in $B$. In particular a minimal (infinite) ideal is homogeneous.\end{proposition}
\begin{proof}Let $X$ be a definable group such that all $B$-chief factors of $A$ are isogenous to $X$. Note that by radicability, all chief factors of $B$ are minimal. Let $B=L_0>L_1>\cdots>L_n=A$ be the beginning of a left chief series for $B$ passing through $A$.
	
\begin{claimm}For $i<n$, if $A$ is multiplicatively normal in $L_{i+1}$, it is multiplicatively normal in $L_i$.\end{claimm}
\begin{claimproof}
Since $L_i/L_{i+1}$ is radicable, it is minimal, and $L_{i+1}$ is a maximal connected strong left ideal in $L_i$. Recall that
$$(x\circ y)\ast z=x\ast(y\ast z)+y\ast z+x\ast z.$$
Consider $x\in L_i$, $a\in A$ and $y\in L_{i+1}$. Then $y_1=x\ast y\in L_{i+1}$, $a_1=a\ast y\in A$, $a_2=a\ast y_1\in A$, $a_3=x^{-1}\ast a_1\in A$ and $x^{-1}\ast a_2\in A$ since $L_{i+1}$ is an ideal in $L_i$ by Corollary \ref{SolBiNil}, and $A$ is an ideal in $L_{i+1}$ and a strong left ideal in $L_i$. Hence:
$$\begin{aligned}(x^{-1}\circ a\circ x)\ast y&=(x^{-1}\circ a)\ast (x\ast y) + x\ast y + (x^{-1}\circ a)\ast y\\
&=(x^{-1}\circ a)\ast y_1 + y_1 + x^{-1}\ast (a\ast y) + a\ast y + x^{-1}\ast y\\
&= x^{-1}\ast (a\ast y_1) + a\ast y_1 + x^{-1}\ast y_1 + y_1 + x^{-1}\ast a_1 + a_1 + x^{-1}\ast y\\
&= x^{-1}\ast a_2 + a_2 + x^{-1}\ast y_1 + y_1 + a_3 + a_1 + x^{-1}\ast y\in A,\end{aligned}$$
as
$$0=(x^{-1}\circ x)\ast y=x^{-1}\ast(x\ast y)+x\ast y+x^{-1}\ast y=x^{-1}\ast y_1+y_1+x^{-1}\ast y$$
and $A$ is additively normal in $L_i$.
Hence $(x^{-1}\circ A\circ x)\ast L_{i+1}\subseteq A$. But $x^{-1}\circ A\circ x$ is a normal multiplicative subgroup of $L_{i+1}$. It follows that $(x^{-1}\circ A\circ x)+A=x^{-1}\circ A\circ x\circ A$ is a skew sub-brace of $L_{i+1}$ containing $A$ and trivial modulo $A$.

Since all quotients in a $B$-left chief series of $A$ are minimal, $X$-$\B$-homogeneity implies $X$-$\{1\}$-homogeneity. But $A$ is strongly left soluble by Corollary \ref{lemsolrk3} and in particular has a series $A=A_0>A_1>\cdots>A_m$ such that
$A_{i+1}$ is an ideal in $A_i$ and $A_i/A_{i+1}$ is minimal trivial for all $i<m$. Then $(A_i/A_{i+1})_\circ=(A_i/A_{i+1})_+\isog X$ for all $i<m$ by Remark \ref{r:series}. In particular $A_\circ$ is $X$-$\{1\}$-homogeneous, as is $(x^{-1}\circ A\circ x\circ A)_\circ$. By triviality $(x^{-1}\circ A\circ x\circ A/A)_\circ=((x^{-1}\circ A\circ x)+ A/A)_+$ is $X$-$\{1\}$-homogeneous; as $A_+$ is $X$-$\{1\}$-homogeneous, so is $((x^{-1}\circ A\circ x)+ A)_+$. It is then contained in a connected $X$-$\{1\}$-homogeneous additive normal subgroup $H$ by Lemma \ref{l:Xhom}. But for any $b\in B$ the image $\lambda_b(H)$ is again normal $X$-$\{1\}$-homogeneous; by finiteness of Morley rank the sum $M:=\sum_{b\in B}\lambda_b(H)$ equals a finite subsum, which is a connected $X$-$\{1\}$-homogeneous strong left ideal by Remark \ref{r:series}. Again minimality of the $B$-chief factors implies that $M$ is $X$-$\B$-homogeneous, so $M=A$ by maximality of $A$. Thus $A$ is multiplicatively normal in $L_i$.
\end{claimproof}
The result follows: As $A$ is multiplicatively normal in $L_n=A$, it is multiplicatively normal in $L_0=B$, i.e.\ an ideal. For the last assertion, if $M$ is a minimal ideal, choose the last $B$-left chief factor $X$ of $M$. Then $X$ is minimal, and the $X$-homogeneous component of $B$ intersects $M$ non-trivially, and must contain $M$ by minimality of $M$. So $M$ is $X$-homogeneous.
\end{proof}

\begin{theorem}
Let $B$ be a connected skew brace of finite Morley rank. Then $B$ is binilpotent if and only if it is left nilpotent of nilpotent type.
\end{theorem}
\begin{proof}
Suppose first that $B$ is left nilpotent of nilpotent type. Then $\B$ is nilpotent by Theorem \ref{t:lnp-Bnp}, so $B_+$ and $B_\circ$ are both nilpotent.

Conversely, assume that $B$ is a counterexample of minimal Morley rank, so both $B_+$ and $B_\circ$ are nilpotent, but $B$ is not left nilpotent. By Corollary \ref{lemsolrk3} any left chief series of $B$ is abelian, and $B$ is strongly left soluble. Moreover $B$ decomposes as a finite sum of bi-$p$ ideals $M_p$  of finite exponent for various primes $p$, and a radicable ideal $D$ by Theorem \ref{t:decomposition}.

Each $\mathfrak M_p$ is a soluble $p$-group, whence locally nilpotent, and nilpotent-by-finite by stability. By connectedness $\mathfrak M_p$ is nilpotent, so $M_p$ is left nilpotent of nilpotent type by Theorem \ref{t:lnp-Bnp}. As a finite sum of left nilpotent ideals of nilpotent type is again left nilpotent of nilpotent type by Lemma \ref{NormalG} and Theorem \ref{t:lnp-Bnp}, we may assume that $B=D$ is radicable.

If $B$ is of left chief length 1, it is trivial, whence left nilpotent. Otherwise let $I<B$ be a maximal connected strong left ideal; by Corollary \ref{SolBiNil} it is actually an ideal. Let $L\le I$ be a minimal connected ideal in $B$. Then $B>I\ge L>\{0\}$. Note that $B/I$ is a trivial minimal brace, and $I$ is left nilpotent as $\RM(I)<\RM(B)$. Let $L=L_0>L_1>\cdots>L_n=\{0\}$ be a $B$-left chief series for~$L$. Note that all factors $L_i/L_{i+1}$ are minimal by Lemma \ref{Minimality}, since $B$ is radicable.

\begin{claim}\label{claimeins}
$I\ast L_i\subseteq L_{i+1}$ for all $i<n$, but there is $j<n$ such that $B\ast L_j\not\subseteq L_{j+1}$. More precisely, $\ker(\lambda^{L_j/L_{j+1}})^0=I$ and there is an algebraically closed field $K$ such that $L_j/L_{j+1}\cong K_+$ and $(B/I)_\circ=(B/I)_+$ isogenically embeds into $K^\times$.
\end{claim}
\begin{claimproof}
$L_i/L_{i+1}$ is minimal, whence $I$-minimal. Hence $I\ast L_i\subseteq L_{i+1}$ by Lemma \ref{BLL_1}.

Suppose $B\ast L_i\subseteq L_{i+1}$ for all $i<n$. Since $B/L$ is left nilpotent by minimality of $\RM(B)>\RM(B/L)$, we see that $B$ is left nilpotent, a contradiction. Hence there is $j<n$ such that $B\ast L_j\not\subseteq L_{j+1}$. Hence $I\le\ker(\lambda^{L_j/L_{j+1}})<B$. By minimality of $B/I$, we have $\ker(\lambda^{L_j/L_{j+1}})^0=I$. The conclusion follows by Theorem \ref{ZilberLin}.
\end{claimproof}
\begin{claim} $L\ast B\not\subseteq L_1$.\end{claim}
\begin{claimproof}
Suppose $L\ast B\subseteq L_1$. Then $L_1$ is an ideal of $B$, so $L_1=\{0\}$ by minimality of $L$. It follows that $L\subseteq \ker(\lambda)$, and that $L$ is a minimal left ideal of $B$, whence $L\le Z^+(B)$ by Lemma \ref{BLL_1}. Moreover $L$ is trivial, whence multiplicatively minimal and normal in $B_\circ$, so $L\le Z^\circ(B)$ by Lemma \ref{center}. Thus, $L\subseteq\ann(B)$. Now $B/L$ is left nilpotent since $\RM(B/L)<\RM(B)$, so $B$ is left nilpotent, a contradiction.
\end{claimproof}
Now $L\ast B\subseteq L$ since $L$ is an ideal, and $L\le Z^+(B/L_1)$ by Lemma \ref{BLL_1}. By Lemma \ref{WelDefstar} for all $\ell\in L$ the map
$$\ast_\ell:x\mapsto \ell\ast x+L_1$$
is a well-defined homomorphism from $B_+$ to $L/L_1$. As $L\ast B\not\subseteq L_1$ and $B_\circ$ is nilpotent, there is a maximal $i$ such that  $\big(Z_{i-1}^\circ(B)\cap L\big)\ast B\subseteq L_1$, and $\ell\in\big(Z_i^\circ(B)\cap L\big)\setminus Z_{i-1}^\circ(B)$ such that $\ast_\ell$ is non-zero. Put $S=\ker(\ast_\ell)^0$.
\begin{claim}
$S$ is a connected proper strong left ideal of $B$ containing $L$.
\end{claim}
\begin{claimproof}
Clearly, $\ker(\ast_\ell)$ is a normal subgroup of $B_+$ containing $L$ since $L\ast L\subseteq I\ast L\subseteq L_1$ by Claim~\ref{claimeins}. But for $b\in B$ and $c\in\ker(\ast_\ell)$ we have $[\ell,b]\in Z_{i-1}^\circ(B)\cap L$ by normality of $L$, so
$$\lambda_\ell\lambda_b(c)=\lambda_b\lambda_\ell\lambda_{[\ell,b]}(c)\in\lambda_b\lambda_\ell(c+L_1)=\lambda_b\lambda_\ell(c)+L_1=\lambda_b(c)+L_1.$$
It follows that $\ker(\ast_\ell)$ is $\lambda$-invariant, whence a strong left ideal in $B$. It is proper by the choice of $\ell$. Therefore $S=\ker(\ast_\ell)^0$ is a connected proper strong left ideal in $B$.
\end{claimproof}

\begin{claim}$S\le I$.\end{claim}
\begin{claimproof}
Suppose not. Then $S\circ I=S+I=B$ since $B/I$ is minimal. Moreover $S$ is left nilpotent since $\RM(S)<\RM(B)$, so the $\lambda$-action of $S$ on all $S$-left chief factors must be trivial by Lemma \ref{BLL_1}. But $L_i/L_{i+1}$ is minimal for all $i<n$, whence an $S$-left chief factor. As the $\lambda$-action of $I$ on those factors is trivial, the $\lambda$-action of $B$ is trivial, contradicting Claim \ref{claimeins}.
\end{claimproof}

As $L/L_1$ is minimal, $\ast_\ell$ is surjective. So $B/S$ is minimal, whence $S=I$. Then $L/L_1$ is isogenous to a subgroup of $K^\times$, and cannot be isogenous to  $L_j/L_{j+1}\cong K_+$ by Lemma \ref{Kequi}. This shows that $L$ is not homogeneous, contradicting Proposition \ref{p:homogeneous}.
\end{proof}

\section{Conclusion and future work}\label{sec:6}

%The results obtained lay the foundations for a more profitable analysis of skew braces of finite Morley rank. 

Skew braces of Morley rank $2$ can actually be characterized as in the case of groups. This will be done in \cite{SkewSMR}.

\smallskip

A class of skew braces that falls outside the scope of our study is the class of those skew braces whose additive or multiplicative group is simple non-abelian. Examples of such skew braces appear when studying skew braces of Morley rank~$3$, since $\operatorname{PSL}_2(K)$ is a simple group of Morley rank $3$ for any algebraically closed field $K$ (up to isomorphism, these are the only non-soluble connected groups of Morley rank $3$ by \cite{frecon}). In \cite{SkewSMR}, we prove that the only non-bisoluble skew braces of Morley rank $3$ are trivial and almost trivial skew braces with additive group isogenous to $\operatorname{PSL}_2(K)$ for an algebraically closed field $K$. A general analysis of skew braces with simple additive/multiplicative group appears to be impossible at the moment.

\smallskip

Another direction of research is whether Theorem \ref{SolLA} can be strengthened in some ways. 

\begin{question}
    Let $B$ be a bisoluble connected skew brace of finite Morley rank of left chief length $3$. Is $B$ strongly left soluble? Or even soluble?
\end{question}

\begin{question}
    Let $B$ be a bisoluble connected skew brace of finite Morley rank. Assume that the left chief length is four. Is $B$ weakly soluble? For greater left chief lengths? And in general?
\end{question}

At present, we can only answer the previous question in the case of Morley rank $4$ (see~\cite{SkewSMR}).

\end{document}